\documentclass[a4paper, 12pt]{amsart}
\numberwithin{equation}{section}
\usepackage{amssymb}
\usepackage{bm}
\usepackage{amsmath}
\usepackage{amsfonts}
\usepackage{marvosym}
\usepackage{tikz}

\newtheorem{thm}{Theorem}[section]
\newtheorem{lem}[thm]{Lemma}

\newtheorem{prop}[thm]{Proposition}

\theoremstyle{definition}
\newtheorem{rem}[thm]{Remark}
\newtheorem{exa}[thm]{Example}
\newtheorem{defin}[thm]{Definition}

\def\ca{{\mathcal A}}
\def\cb{{\mathcal B}}

\def\ce{{\mathcal E}}
\def\cf{{\mathcal F}}

\def\ch{{\mathcal H}}

\def\ck{{\mathcal K}}
\def\cl{{\mathcal L}}

\def\cn{{\mathcal N}}

\def\cp{{\mathcal P}}

\def\cs{{\mathcal S}}

\def\cu{{\mathcal U}}

\def\cz{{\mathcal Z}}

\def\ga{{\mathfrak A}}
\def\gb{{\mathfrak B}}

\def\gz{{\mathfrak Z}}

\def\bc{{\mathbb C}}

\def\bm{{\mathbb M}}
\def\bn{{\mathbb N}}
\def\bp{{\mathbb P}}

\def\bz{{\mathbb Z}}

\def\a{\alpha}

\def\eeps{\epsilon}
\def\eps{\varepsilon}

\def\m{\mu}

\def\r{\rho}
\def\s{\sigma} 

\def\f{\varphi}  
\def\th{\theta} 
\def\om{\omega}

\def\id{\hbox{id}}

\def\di{{\rm d}}
\def\id{{\rm id}}

\def\ad{\mathop{\rm ad}}

\DeclareMathAlphabet{\mathpzc}{OT1}{pzc}{m}{it}

\begin{document}

\title[de Finetti-type theorems]
{de Finetti-type theorems on quasi-local algebras and infinite Fermi tensor products}
\author{Vitonofrio Crismale}
\address{Vitonofrio Crismale\\
Dipartimento di Matematica\\
Universit\`{a} degli studi di Bari\\
Via E. Orabona, 4, 70125 Bari, Italy}
\email{\texttt{vitonofrio.crismale@uniba.it}}
\author{Stefano Rossi}
\address{Stefano Rossi\\
Dipartimento di Matematica\\
Universit\`{a} degli studi di Bari\\
Via della E. Orabona, 4, 70125 Bari, Italy} \email{{\tt
stefano.rossi@uniba.it}}
\author{Paola Zurlo}
\address{Paola Zurlo\\
Dipartimento di Matematica\\
Universit\`{a} degli studi di Bari\\
Via E. Orabona, 4, 70125 Bari, Italy}
\email{\texttt{paola.zurlo@uniba.it}}

\begin{abstract}
Local actions of $\bp_\bn$, the group of finite permutations on $\bn$, on quasi-local
algebras are defined and proved to be $\bp_\bn$-abelian.
It turns out that invariant states under local actions are
automatically even, and extreme invariant states are strongly clustering.
Tail algebras of invariant states are shown to obey 
a form of the Hewitt and Savage theorem, in that they coincide with the 
fixed-point von Neumann algebra.
Infinite graded tensor products of $C^*$-algebras, which  include the CAR algebra,
are then addressed as particular examples of quasi-local algebras 
acted upon $\bp_\bn$ in
a natural way. Extreme invariant states are characterized as infinite products of a single even state, and
a de Finetti theorem is established.
Finally, infinite products of factorial even states are shown to be factorial
by applying a twisted version of  the tensor product commutation theorem, which is also derived here.

\end{abstract}

\maketitle
\section{Introduction}\label{sec1}

Distributional symmetries for families of random variables concern  invariance of any finite joint distribution of them under some measurable transformations. For their importance in probability theory,  invariance under shifts, finite permutations or rotations are certainly worth mentioning. In these cases the random variables are respectively named stationary, exchangeable or rotatable, and the reader is referred to \cite{Ka} for an extensive account of the subject in the setting of commutative probability spaces.
The investigation of distributional symmetries was initiated by de Finetti's celebrated theorem, which shows that  sequences of two-point valued exchangeable random variables are conditionally independent and identically distributed. Phrased differently, any finite joint distribution of them is obtained by randomization of the binomial distribution. This result has since
 found several generalizations. To name but one of these, the probability measures on the Tychonov product of compact Hausdorff spaces which are invariant under the action of finite permutations are in fact mixtures of product measures, as
proved by Hewitt and Savage in \cite{HS}.  

Now the $C^*$-algebraic counterpart of Tychonov products is provided by the theory of tensor products of
$C^*$-algebras. Therefore, it is no wonder that the earliest non-commutative  settings for the generalizations of de Finetti's theorem
came from the infinite tensor products of a given unital $C^*$-algebra. 
In \cite{St}, St\o rmer  carried out a thorough analysis of all permutation-invariant states of the (minimal) infinite
tensor product  $\otimes^{\bn}\ga$ of an assigned $C^*$-algebra $\ga$.
Among the main results obtained in that paper, it is worthwhile to mention that the extreme points of the (weakly-$*$) compact convex set of such states 
 may be identified with infinite product states of a single state on $\ga$. Furthermore, the convex set in question is actually
a Choquet simplex, which allows for a decomposition of any invariant state into an integral of extreme invariant states with respect to
a unique barycentric measure.
To our knowledge, though, it was not until the early $90$s
that this line of research got a new lease of life, when  far more emphasis was laid on
the probabilistic interpretation.  In this respect, Accardi and Lu
proved a general non-commutative version of the Hewitt and Savage theorem,  \cite{AL}.
In a later paper, \cite{ABCL}, connections between exchangeability and singleton
conditions were also established.
Not long after, K\"{o}stler obtained a non-commutative de Finetti theorem within
the formalism of von Neumann algebras in \cite{K}, where exchangeability
is seen to imply independence with respect to the tail algebra, although the converse
may fail to hold, as remarked by the author himself.  Finally, also motivated by the key role played in physics by the 
canonical anti-commutation rules, Crismale and Fidaleo provided a version of the theorem for states right on the
CAR algebra, \cite{CFCMP}.
Although the CAR algebra is isomorphic with the UHF algebra of type $2^\infty$ and is thus an infinite
tensor product of  $\bm_2(\bc)$ with itself, the de Finetti theorem proved in the last mentioned paper cannot
be reached  by an application of the results in \cite{St}, not least because the action of the permutations is not the same
as the one considered by St\o rmer.
In fact, the results obtained there  take into account the canonical $\bz_2$-grading of
the CAR algebra as well. In particular, any symmetric state turns out to be even, namely grading-invariant.
Furthermore, extreme symmetric states feature the same properties as in the work of Stormer. The 
novelty, however, is that the product must be intended in the sense of Araki and Moriya, \cite{AM1}, and the factor
state must be an even state on $\bm_2(\bc)$, thought of as a graded $C^*$-algebra with even (odd) part
given by diagonal  (anti-diagonal) matrices, for the product state to even make sense. Unlike  what happens with usual
tensor products, the action of $\bp_\bn$, the group of finite permutations on $\bn$, on the CAR algebra is no longer asymptotically abelian.
Nevertheless,  the corresponding $C^*$-dynamical system is $\bp_\bn$-abelian,  see \cite{Sak} for the definition.
It is ultimately this circumstance which  guarantees that the set of symmetric states is still a Choquet simplex.

This paper in part aims to resume the analysis carried out in \cite{CFCMP} in order to frame it in the broader scope
of quasi-local algebras, the interest in which is undoubtedly justified by the many appearences they make in 
quantum field theory and statistical mechanics. 
In the present work, however, quasi-local algebras are mainly thought of as a source of examples of $\bz_2$-graded $C^*$-algebras.
In particular, in Section \ref{quasilocal} we first single out  actions of $\bp_\bn$ which are fully
compatible with the local structure of the algebras addressed, see Definition \ref{locact}.
More in detail, Proposition \ref{symmquasi} shows that any such action is $\bp_\bn$-abelian. 
Moreover, its invariant states are automatically even, with extreme states
being weakly clustering. These are then shown to be strongly clustering in Theorem
\ref{extremelocal}.  Tail algebras of invariant
states  are then given a good deal of attention. In Proposition \ref{tailtriv} we show that the tail algebra of an extreme invariant state is
always trivial. Tail algebras corresponding to non-extreme invariant states, too, can be analized in full detail.
In the first place, their structure is disciplined by a form of the Hewitt and Savage theorem,
in so far as they coincide with the $\bp_\bn$-invariant part of the center of the von Neumann algebra
generated by the given state. As a consequence, they are always abelian and
decompose into a direct integral of ergodic components, as proved in
Proposition \ref{localtail}. The section ends with Proposition \ref{deFinetti}
and Proposition \ref{deFinetti2}, which provide  de Finetti-type theorems for nets of local algebras.
In particular, under the assumption of additivity of the net,
Proposition \ref{deFinetti2} characterizes symmetry of states in terms of a condition reminiscent of identical distribution, and conditional independence of the local algebras
with respect to the conditional expectation onto the tail algebra.\\
Section \ref{infprod} is devoted to infinite $\bz_2$-graded tensor products as distinguished and particularly well-behaved
instances of quasi-local algebras. After providing a quick exposition of infinite graded tensor products, we show in Example \ref{CAR} how the CAR
algebra can be recovered as a suitable infinite product of this type.
The group of finite permutations acts in a natural way on infinite graded tensors products.  
Invariant states for this action lend themselves to a more accomplished description
as opposed to the case of quasi-local algebras. In particular, extreme states can be identified with infinite
products of a single even state, Proposition \ref{extremeprod}. Moreover, as shown in Proposition \ref{weakerg}, the action also turns out to be weakly
ergodic when it is the minimal product to be dealt with. Finally, infinite graded tensor products offer
quite a natural setting to state a fully-fledged version of de Finetti's theorem, for in this case
invariant states correspond to exchangeable quantum stochastic processes, see also \cite{CrFid, CrFid2}.
This is done in Theorem \ref{ffDeFinetti}, where such processes are characterized in terms of identical distribution
and conditional independence.\\
In Section \ref{twistedcomm} we further develop the analysis 
of infinite product states by showing pureness (factoriality) when each
factor is pure (factorial), Proposition \ref{pure} (Proposition \ref{factor}).
The proof of these results heavily relies on a twisted version of the well-known
tensor product commutation theorem, which is obtained in Theorem \ref{VanDaele} 
for the tensor product of two
graded von Neumann algebras, and in Theorem \ref{infinitecomm} for the tensor product of infinitely many von Neumann algebras. Finally, the analysis by
St\o rmer in \cite{St} on the type of factor one can obtain from 
the GNS representation of product states applies to the present 
framework, and the relative results are gathered in 
Proposition \ref{typefactor}.

\section{Symmetric states on quasi-local algebras}\label{quasilocal}

By a $\bz_2$-graded $C^*$-algebra we mean a pair $(\ga, \theta)$ made up of
a (unital) $C^*$-algebra and a (unital) $*$-automorphism $\theta$ which is involutive, namely
$\theta^2={\rm id}_\ga$. Setting $\ga_{1}:=\{a\in\ga: \theta(a)=a\}$ and
$\ga_{-1}:=\{a\in\ga: \theta(a)=-a\}$, one easily sees that $\ga$ decomposes as
$$
\ga=\ga_1\oplus\ga_{-1}\,,
$$
where the direct sum is topological, and
$$
(\ga_i)^*=(\ga^*)_i\,,\,\,
\ga_i\ga_j\subset\ga_{ij}\,,\quad i,j=1,-1\,.
$$
Note that $\ga_1$ is a (unital) $C^*$-subalgebra of $\ga$, while $\ga_{-1}$ is only an involutive closed subspace of $\ga$.
The subspaces $\ga_i$, $i=1,-1$ are often referred to as the homogeneous components of $\ga$, and correspondingly any element of $\ga_i$ is called a homogeneous element of $\ga$.
For any homogeneous element $x\in\ga_{\pm 1}$ we denote its {\it grade} by
$$
\partial(x)=\pm1.
$$
It is easy to see that considering an involutive $*$-automorphism $\theta$ on $\ga$
amounts to assigning a decomposition of $\ga$ into a topological direct sum as above.
Indeed, if one is given such a decomposition, then the corresponding automorphism $\theta$ can be defined as
\begin{equation*}
\th\lceil_{\ga_1}:=\id_{\ga_1}\,,\quad \th\lceil_{\ga_{-1}}:=-\id_{\ga_{-1}}\,.
\end{equation*}
Note that
$$
\eps_\theta:=\frac{1}{2}(\id_{\ga}+\th)\,,
$$
defines a faithful conditional expectation onto $\ga_1$.
When there is no risk of confusion, we will suppress the underscript from $\eps_\theta$ and simply
write $\eps$.
The  $*$-subalgebra $\ga_+:=\ga_1$  and the subspace $\ga_-:=\ga_{-1}$ are commonly referred to as  the {\it even part} and the {\it odd part} of $\ga$, respectively.
Clearly, any $a\in\ga$ can be written as a sum $a=a_++a_-$, with $a_+\in\ga_+$, $a_-\in\ga_-$,
and this decomposition is unique.
Taking $\th=\id_{\ga}$, one sees that any $*$-algebra $\ga$ is equipped with a $\bz_2$ trivial grading. Here, $\ga_+=\ga$ and $\ga_-=\{0\}$.\\
A simple example of $\bz_2$-graded $*$-algebra is obtained by taking a Hilbert space $\ch$, and a bounded self-adjoint unitary $U$ on $\ch$. The adjoint action $\ad_{U}(\cdot):=U \cdot U^*$ is an involutive $*$-automorphism which induces a $\bz_2$-grading on $\cb(\ch)$.

Let $\big(\ga_i,\th_i\big)$, $i=1,2$, be two  $\bz_2$-graded $*$-algebras.
The map $T:\ga_1\to\ga_2$ is said to be {\it even} if it is grading-equivariant, {\it i.e.}
$$
T\circ\th_1=\th_2\circ T\,.
$$
When $\th_2=\id_{\ga_2}$, the map $T:\ga_1\to\ga_2$ is even if and only if it is grading-invariant, that is
$T\circ\th_1=T$. If $T$ is $\bz_2$-linear,
then it is even if and only if $T\lceil_{\ga_{1,-}}=0$. When $\big(\ga_2,\th_2\big)=\big(\bc,\id_\bc\big)$, a functional $f:\ga_1\to\bc$ is even if and only if $f\circ\th=f$.

In the sequel, we will denote by $\cs_+(\ga)$ the weakly-$*$ compact convex  subset of all even states. Even states play a role in giving a $\bz_2$-grading to their GNS structures.
More in detail, suppose that $(\ga,\th)$ is a $\bz_2$-graded $C^*$-algebra, and $\f\in\cs_+(\ga)$. Let $(\ch_\f,\pi_\f, \xi_\f, V_{\th,\f})$ be the GNS covariant representation of $\f$, where the unitary self-adjoint $V_{\th,\f}$ fixes $\xi_\f$ and verifies
$$
\pi_\f(\th(a))=V_{\th,\f}\pi_\f(a)V_{\th,\f}\,, \quad a\in\ga\,.
$$
Then, $(\cb(\ch),\ad_{V_{\th,\f}})$ is a $\bz_2$-graded $C^*$-algebra.
If $\f$ is a pure state, though, evenness is no longer necessary for  
 a unitary on $\ch_\varphi$ implementing the grading to exist. In fact, all is needed is that
$\pi_\varphi$ and $\pi_{\varphi\circ\theta}$ are not disjoint representations\footnote{This means
that there exists a non-null intertwining operator $T$, {\it i.e.} a
$0\neq T\in \cb(\ch_\f, \ch_{\f\circ\theta})$ such that
$T\pi_\f(a)=\pi_{\f\circ\theta}(a) T$ for all $a\in\ga$.}. More
precisely, one has the following.

\begin{prop}\label{ArMor}
Let $\varphi\in\cs_+(\ga)$ be a pure state such that
$\pi_\varphi$ and $\pi_{\varphi\circ\theta}$ are not disjoint. Then there exists a self-adjoint
unitary $U\in \pi_\f(\ga_+)''$ such that
$$U\pi_\f(a)U^*=\pi_\f(\theta(a)), \,\, a\in\ga\quad {\rm and}\,\, \langle U\xi_\f, \xi_\f\rangle\geq 0\,.$$
\end{prop}

\begin{proof}
Same proof as Lemma 3.1 in \cite{AM1}.
\end{proof}

We can now move on to consider quasi-local algebras
as notable examples of $\bz_2$-graded $C^*$-algebra.
To this aim, denote by $\cp_0(\bn)$ the set of all finite subsets of $\bn$.

\begin{defin}\label{lnet}
By a quasi-local algebra over  $\cp_0(\bn)$ we mean a unital $\bz_2$-graded $C^*$-algebra
$(\ga, \theta)$, where $\ga$ is the inductive limit of a
net $\{\ga(I): I\in\cp_0(\bn)\}$ of local unital $C^*$-subalgebras $\ga(I)\subset\ga$ such that:

\medskip
\begin{itemize}
\item [(i)]  for every $I, J\in \cp_0(\bn)$ with $I\subset J$, one has $\ga(I)\subset\ga(J)$;
\item [(ii)] for every $I\in\cp_0(\bn)$, one has $\theta(\ga(I))=\ga(I)$;
\item [(iii)] for every $I, J\in\cp_0(\bn)$ with $I\cap J=\emptyset$, and homogeneous $x\in\ga(I)$ and $y\in\ga(J)$,
 $x$ and $y$ commute when one of them is even, and anticommute when they are both odd.
\end{itemize}
\end{defin}

We should mention that the net of local algebras can of course be indexed by more general sets than $\cp_0(\bn)$, see
{\it e.g.} \cite{BR1}. However, the choice of $\cp_0(\bn)$ made here is the most appropriate insofar as we want
our quasi local-algebra to be acted upon by $\bp_\bn$, the group of finite permutations of $\bn$.
More  precisely, throughout this section we will be focusing on local actions of $\bp_\bn$ on $\ga$, as
defined below.

\begin{defin}\label{locact}
A local action of $\bp_\bn$ on  a quasi-local $C^*$-algebra $\ga$ is a group homomorphism
$\a: \bp_\bn\rightarrow{\rm Aut}(\ga)$ such that
\begin{itemize}
\item [(i)] the action is grading-equivariant, that is $\a_\s\circ\theta=\theta\circ\a_\s$, for every
permutation $\s\in\bp_\bn$;
\item [(ii)]  for every finite subset $I\subset\bn$ and $\s\in\bp_\bn$, one has $\a_\s(\ga(I))=\ga(\s(I))$.
\end{itemize}
\end{defin}

We next show that the states of $\ga$ which are invariant under such an action
of $\bp_\bn$ enjoy good properties. First, they are automatically even.
Second, they are weakly (and in fact strongly) clustering as soon as they are extreme.
These properties are proved in the propositions below. Before stating them, though, some notation and definitions
need to be set first. \\
A state $\om$ on $\ga$ is invariant under $\a$, or equivalently $\a$-invariant, if
$\om\circ\a_\s=\om$, for every $\s\in\bp_\bn$. The set of all $\a$-invariant states, which we denote
by $\cs^{\bp_\bn}(\ga)$, is weakly-$*$
compact and convex. Its extreme states are called the {\it ergodic} states for the action
of $\bp_\bn$. The set of all invariant extreme states
will be denoted by $\ce(\cs^{\bp_\bn}(\ga))$.
We will be using the terms invariant states and symmetric states
interchangeably throught the paper.\\
If now $(\ch_\om, \pi_\om, \xi_\om)$ is the GNS triple associated with a given state $\om$ in $\cs^{\bp_\bn}(\ga)$, the action of every
$\a_\s$ can be implemented on the Hilbert space $\ch_\om$ by a unitary $U_\s^\om$ uniquely determined by
$$U_\s^\om \pi_\om(a)\xi_\om:=\pi_\om(\a_\s(a))\xi_\om,\quad a\in\ga\,.$$
We denote by $\ch_\om^{\bp_\bn}\subset\ch_\om$ the closed subspace of all invariant vectors under the action of the unitaries
$U_\s^\om$, namely
$$\ch_\om^{\bp_\bn}:=\{\xi\in\ch_\om: U_\s^\om \xi=\xi,\, \textrm{for all}\, \s\in\bp_\bn\}\, .$$
The orthogonal projection onto $\ch_\om^{\bp_\bn}$ is denoted by $E_\om$.
As is clear, the one-dimensional subspace $\bc\xi_\om$
is contained in $\ch_\om^{\bp_\bn}$, which means
$E_\om$ is never $0$.
As is known from the general theory of group actions through
automorphisms on $C^*$-algebras,  the condition that
$\ch_\om^{\bp_\bn}$ reduces to $\bc\xi_\om$ implies that
$\om$ is extreme in $\cs^{\bp_\bn}(\ga)$, see {\it e.g.} \cite[Proposition 3.1.10]{Sak}.
The reverse implication may well fail to hold for a given action of a given group $G$ on a general $C^*$-algebra $\ga$.
However, it does hold provided that the system $(\ga, G, \a)$ is what is known as a $G$-abelian dynamical system. This is by definition the case
when, for every $G$-invariant state $\om$, the set $E_\om\pi_\om(\ga)E_\om$ is an abelian family of operators acting on
$\ch_\om$.
Among other things, we next show that any local action
of $\bp_\bn$ on a quasi-local $C^*$-algebra is
always $\bp_\bn$-abelian.

For every natural $n$, denote by $\bp_n\subset\bp_\bn$ the finite
subgroup of permutations which act trivially from $n+1$ onwards.
Note that $\bp_\bn=\bigcup_n \bp_n$.
We adopt the notation in \cite{CFCMP} and define the Cesàro average of an arbitrary operator-valued function
$f: \bp_\bn\rightarrow \cb(\ch)$ as
$$M(f(\s)):=\lim_{n\rightarrow\infty} \frac{1}{n!}\sum_{\s\in \bp_n} f(\sigma)$$
as long as the limit exists in a suitable sense (for instance in the strong/weak operator topology).
We recall that for any $\om$ in $\cs^{\bp_\bn}(\ga)$ and $\s\in\bp_\bn$ one has
$M(U_\s^\om)=E_\om$, where the equality is understood in the strong operator topology,
see \cite[Proposition 3.1]{CFCMP}.

\begin{prop}\label{symmquasi}
Let $\a$ be a local action of $\bp_\bn$ on a quasi-local algebra $\ga$. If $\om\in\cs^{\bp_\bn}(\ga)$, then
\begin{enumerate}
\item $\om$ is even;
\item $E_\om\pi_\om(\ga) E_\om$ is a commuting family of operators, hence the dynamical system is
$\bp_\bn$-abelian;
\item $\om\in\ce(\cs^{\bp_\bn}(\ga))$  if and only if
${\rm dim} \ch_\om^{\bp_\bn}=1$.
\end{enumerate}
\end{prop}

\begin{proof}
As for (1), we need to show that any symmetric state $\om$ vanishes on all odd elements of $\ga$.  By density, it is
enough to prove that $\om(a)=0$ for every $a$ which is a localized odd element, say
$a\in\ga(I)$ for some finite subset $I\subset\bn$. Denoting by $\{\cdot, \cdot\}$
the anticommutator, for an $a$ as before we have
\begin{align*}
&\{E_\om\pi_\om(a) E_\om, E_\om\pi_\om(a^*)E_\om\}\\
=& M(E_\om\pi_\om(a)U_\s^\om\pi_\om(a^*)E_\om+E_\om\pi_\om(a^*)U_\s^\om\pi_\om(a)E_\om)\\
=&M(E_\om\pi_\om(\{a, \a_\s(a^*)\})E_\om)\\
=&\lim_{n\rightarrow\infty}\frac{1}{n!}\sum_{\s\in\bp_n}E_\om\pi_\om(\{a, \a_\s(a^*)\})E_\om=0
\end{align*}
where the last equality holds because for every $n$ such that $I\subset\{1, \ldots, n\}$ one has
$|\{\s\in\bp_n: \s(I)\cap I\neq \emptyset\}|\leq C(n-1)!$, with $C$ being a constant that does not depend on
$n$, see \cite[Lemma 3.3]{CFCMP}, whereas if $\s$ is such that $\s(I)\cap I= \emptyset$ then
$\{a, \a_\s(a^*)\}=0$ by virtue of $(iii)$ of Definition \ref{lnet}. This readily implies
that
\begin{equation}\label{oddvanish}
E_\om\pi_\om(a) E_\om=0\quad \textrm{for any odd}\,\, a\in\ga\,.
\end{equation}
In particular, for such an $a$ one has $$\om(a)= \langle \pi_\om(a)\xi_\om, \xi_\om \rangle= \langle E_\om\pi_\om(a)E_\om\xi_\om, \xi_\om \rangle=0\, ,$$ and so (1) is proved.\\
As for (2), thanks to Equality \eqref{oddvanish} it is enough to verify that the commutator
$[E_\om\pi_\om(a) E_\om, E_\om\pi_\om(a) E_\om]$ is $0$ for even $a,b\in\ga$, which can be seen with
similar computations to those in $(1)$.\\
Property (3) holds thanks to Proposition 3.1.12 in \cite{Sak}.
\end{proof}

For any fixed integer $n\geq 1$, we denote by $\s_n$ the permutation acting on $\bn$ as
\begin{equation}\label{sigman}
\s_n(k)=\left\{\begin{array}{lll}
k+2^{n-1}, &  1\leq k\leq 2^{n-1}\\
k-2^{n-1}, &  2^{n-1}< k\leq 2^{n}\\
k, & k> 2^{n}\\
\end{array}\right.
\end{equation}
The next result is key to further characterize extreme symmetric states.
\begin{lem}
 If $\om\in\cs(\ga)$ is an extreme
symmetric state with respect to a local action $\a$ of $\bp_\bn$ on a quasi-local algebra $\ga$, then for every $a\in\ga$ one has
$$\lim_{n\rightarrow\infty} \pi_{\om}(\alpha_{\s_n}(a))\xi_\om=\om(a)\xi_\om, $$
in the weak operator topology.
\end{lem}

\begin{proof}
The proof is the same as in Lemma 5.2 in \cite{CFCMP}.
The only thing that needs to be taken care of is that
if $a\in \ga(I)$ for some  finite subset $I\subset\bn$ and $\s\in\bp_\bn$, there exists
$N_{\s, a}\in\bn$ such that $\alpha_{\s\s_n}(a)=\alpha_{\s_n}(a)$
for every $n\geq N_{A, \s}$. To this end, let $r, s\in\bn$
such that $I\subset \{1, \ldots, r\}$ and the restriction of
$\s$ to $\{n\in\bn: n\geq s\}$ is the identity.
Set $N_{\s, a}:=\max\{r, s\}$. For
$n\geq N_{\s, a}$, we have
$\alpha_{\s\s_n}(a)=\alpha_\s(\alpha_{\s_n}(a))=\alpha_{\s_n}(a)$ because
$\alpha_\s$ acts as the identity on each $\ga(J)$ if $J\cap \{1, \ldots s-1\}=\emptyset$.
\end{proof}
Before stating the announced characterization, we recall that a symmetric state $\om$ is
strongly clustering (or mixing) if  for every $a, b\in\ga$ one has $\lim_n \om(\a_{\s_n}(a)b)=\om(a)\om(b)$, {\it cf.}
\cite{St}.

\begin{thm}\label{extremelocal}
For a symmetric state $\om$ on  a quasi-local algebra $\ga$ acted upon
$\bp_\bn$ through a local action $\a$ the following conditions are equivalent:
\begin{enumerate}
\item $\om$ is extreme;
\item  $\om $ is strongly clustering;
\item  $\om(ab)=\om(a)\om(b)$ for every $a\in\ga(I)$ and $b\in\ga(J)$ and finite subsets $I, J\subset\bn$ such that $I\cap J=\emptyset$.
\end{enumerate}
\end{thm}

\begin{proof}
The equivalence $(1)\Leftrightarrow(2)$ can be proved exactly as is done in \cite[Theorem 5.3]{CFCMP}.
The implication $(3)\Rightarrow (2)$ is obvious, so it remains to show that
$(2)\Rightarrow (3)$. To this aim, consider $\s\in\bp_\bn$ such that
$\s$ is the identity on $I$ and coincides with $\s_m$ on $J$,
where $\sigma_m$ is the permutation defined in \eqref{sigman}. We have
\begin{align*}
\om(ab)=\om(\a_\s(ab))=\om(a\a_{\s_m}(b))=\lim_m\om(a\a_{\s_m}(b))=\om(a)\om(b)\, ,
\end{align*}
where in the second-last equality we have used that $\{\om(a\a_{\s_m}(b))\}_{m\in\bn}$ is actually a constant sequence.
\end{proof}
The next result shows that $\bp_\bn$ is represented by a large group of automorphisms in the sense of \cite{St67}
whenever it acts on a quasi-local $C^*$-algebra as in Definition \ref{locact}.
This means  that for any invariant state $\om\in\cs(\ga)$ and any self-adjoint $a\in\ga$ one has
$$\overline{{\rm conv}}\{\pi_\om(\a_\s(a)): \s\in\bp_\bn\}\cap\pi_\om(\ga)'\neq \emptyset\, .$$

\begin{prop}\label{large}
Any local action $\a$ of $\bp_\bn$ on a quasi local $C^*$-algebra $\ga$ is
a large group of automorphisms.
\end{prop}

\begin{proof}
The proof can be done as in \cite[Theorem 4.2]{CFCMP} once we have first established asymptotic
abelianness in average of any symmetric state. More explicitly, we need to show that
if $\om$ is a symmetric state on $\ga$, then
$M\{ \om(c[\a_\s(a), b]d)\}=0$ for every $a, b, c, d\in\ga$.\\
We start by observing that
\begin{align*}
M\{ \om(c \a_\s(a) bd)\}=&M\{ \om(\a_\s(a_+) cbd)\}\\
&+M\{ \om(\a_\s(a_-)(c_+- c_-) bd)\}
\end{align*}
 as follows by applying \cite[Lemma 3.3]{CFCMP}.
Now the second summand in the right-hand side of the equality above is $0$ since
$E_\om\pi_\om(a_-)E_\om=0$ thanks to \eqref{oddvanish}.
By $\bp_\bn$-abelianness we then have
\begin{align*}
M\{ \om(c \a_\s(a) bd)\}&=M\{ \om(\a_\s(a_+) cbd)\}\\&= \langle \pi_\om(a_+)E_\om\pi_\om(cbd)\xi_\om, \xi_\om\rangle\\
&=\langle\pi_\om(cbd) E_\om \pi_\om(a_+)\xi_\om, \xi_\om\rangle\\&=M\{ \om(cbd\a_\s(a_+))\}\\
&=M\{ \om(cb\a_\s(a_+)d)\}\,,
\end{align*}
where the last equality is due again to \cite[Lemma 3.3]{CFCMP}.
Now, arguing as above, one easily sees that $M\{\om(cb\a_\s(a_-)d)\}=0$, which ends the
proof.
\end{proof}

As the dynamical system $(\ga,\bp_\bn)$ is $\bp_\bn$-abelian, by \cite[Theorem 3.1.14]{Sak}, one has that the set of symmetric states is indeed a Choquet simplex. This means that any $\bp_\bn$-invariant state is the barycenter of a unique probability measure which is pseudo-supported on the set of extreme states, see \cite{BR1}, page $322$. More in detail, we have

\begin{prop}
Let $\a$ be a local action of $\bp_\bn$ on a quasi-local $C^*$-algebra $\ga$.
If $\om\in\cs^{\bp_\bn}(\ga)$, then there exists a unique probability measure $\m$ pseudo-supported on $\ce(\cs^{\bp_\bn}(\ga))$ such that
\begin{equation}\label{bar}
\om(a)=\int_{\ce(\cs^{\bp_\bn}(\ga))}\psi(a)\di\m(\psi)\,, \quad a\in\ga\,.
\end{equation}
\end{prop}

We recall that with any state $\om$ on a quasi-local algebra $\ga$ it is possible to associate a von
Neumann algebra $\gz_\om^\perp\subset \cb(\ch_\om)$ defined as
$$\gz_\om^\perp=\bigcap_{n=1}^\infty\underset{I\in \cf_n}{\bigvee} \pi_\om(\ga(I))''\, ,$$
where $\cf_n$ collects all the finite subsets $I\subset \bn$ such that $I\subset\{n, n+1, \ldots\}$.
This algebra is commonly known as the tail algebra of the state $\om$, although in
quantum statistical mechanics is typically referred to as the algebra at infinity, see also \cite[Definition 2.6.4] {BR1}.
The tail algebra of an ergodic symmetric state is shown to be trivial below.

\begin{prop}\label{tailtriv}
Let $\a$ be a local action of $\bp_\bn$ on a quasi-local $C^*$-algebra $\ga$.
The tail algebra $\gz_\om^\perp$ of any $\om$ in $\ce(\cs^{\bp_\bn}(\ga))$ is trivial.
\end{prop}

\begin{proof}
We first show that $\gz_\om^\perp$ is contained in the fixed-point von Neumann
algebra $\{T\in \cb(\ch_\om): U_\s^\om T =TU_\s^\om,\, \s\in\bp_\bn\}$, where
$U_\s^\om$ is the unitary implementator of $\a_\s$ in $\ch_\om$, {\it i.e.} $U_\s^\om \pi_\om(a)\xi_\om=\pi_\om(\a_\s(a))\xi_\om$, $a\in\ga$.
Let $T$ be in $\gz_\om^\perp$ and $\s\in\bp_\bn$. Then there exists $n_o$ such that $\s$ acts trivially on
$\{n_o, n_o+1, \ldots,\}$. In particular, ${\rm ad}_{U_\s^\om}$
acts trivially on $\pi_\om(\ga(I))$ for every finite subset $I$ contained in $\{n_o, n_o+1, \ldots\}$. As
a result, ${\rm ad}_{U_\s^\om}$ still acts trivially on
$\underset{I\in \cf_{n_o}}{\bigvee} \pi_\om(\ga(I))''$. Now since
$T$ sits in particular in $\underset{I\in \cf_{n_o}}{\bigvee} \pi_\om(\ga(I))''$, we must have
$U_\s^\om T =T U_\s^\om$.\\
Furthermore, by Theorem 2.6.5 in \cite{BR1} we also have that $\gz_\om^\perp$ is contained in
in $\pi_\om(\ga)'\cap \pi_\om(\ga)''$. In particular, $\xi_\om$ is separating for $\gz_\om^\perp$.
Indeed, from $\gz_\om^\perp\subset \pi_\om(\ga)'$ we see $\pi_\om(\ga)''\subset(\gz_\om^\perp)'$, hence
$\xi_\om$ is cyclic for $(\gz_\om^\perp)'$.\\
We are ready to reach the conclusion. Indeed, if $T$ lies in $\gz_\om^\perp$, then
$T\xi_\om$ is an invariant vector. By extremality of $\om$ and $\bp_\bn$-abelianness, we then have
$T\xi_\om=\lambda \xi_\om$ for some $\lambda\in\bc$, which means
$T=\lambda 1$ since $\xi_\om$ is separating for such a $T$.
\end{proof}

The next proposition provides a quantum analogue of the well-known  Hewitt and Savage theorem that
the tail and the symmetric $\s$-algebras of an exchangeable
sequence of random variables actually coincide, see \cite{HS}.  For the reader's convenience we recall that
a sequence of random variables is exchangeable if the joint distribution of any finite subset of variables
is invariant under permutations.\\

Given $\om$ in $\cs^{\bp_\bn}(\ga)$, we set
$Z_{\bp_\bn}(\om):= \cz(\pi_\om(\ga)'')\cap U_\om(\bp_{\bn})'$, where
$$U_\om(\bp_{\bn})':=\{T\in \cb(\ch_\om): U_\s^\om T =TU_\s^\om,\, \s\in\bp_\bn\}$$ and
$ \cz(\pi_\om(\ga)''):=\pi_\om(\ga)''\cap \pi_\om(\ga)'$ is the center of $\pi_\om(\ga)''$.
The vector state $\langle\cdot\xi_\om, \xi_\om\rangle$ on $\pi_\om(\ga)''$ will be denoted by
$\varphi_{\xi_\om}$.

\begin{prop}\label{localtail}
The tail algebra $\gz^\perp_\om$ of a symmetric state $\om\in\cs(\ga)$ coincides with
$Z_{\bp_\bn}(\om)$.\\
 Moreover, there exists a unique conditional expectation
$E_\om: \pi_\om(\ga)''\rightarrow \gz^\perp_\om$. This is given by

$$
E_\om(X)=\int^\oplus_{\ce(\cs^{\bp_\bn}(\ga))}\langle X_\psi\xi_\psi, \xi_\psi\rangle {\rm d}\mu(\psi), \quad X\in \pi_\om(\ga)''\,,
$$
where $\mu$  is the measure appearing in \eqref{bar}, and  $X=\int^\oplus_{\ce(\cs^{\bp_\bn}(\ga))} X_\psi {\rm d}\mu(\psi)$.
In addition, $E_\om$ preserves the vector state
$\f_{\xi_\om}$.
\end{prop}

\begin{proof}
By applying Theorem 4.4.3 in \cite{BR1} and Proposition 3.1.10 in \cite{Sak}, we see that the abelian von Neumann algebra $Z_{\bp_\bn}(\om)$ decomposes into a direct integral as
$$Z_{\bp_\bn}(\om)=\int^\oplus_{\ce(\cs^{\bp_\bn}(\ga))}\bc 1_{\ch_\psi} {\rm d}\mu(\psi)\cong L^\infty(\ce(\cs^{\bp_\bn}(\ga)), \mu)\, .$$
Because the diagonal operators of $\int^\oplus_{\ce(\cs^{\bp_\bn}(\ga))} \ch_\psi {\rm d}\mu(\psi)$  are contained in $\pi_\om(\ga)''$, we can apply Lemma 8.4.1 in \cite{Dix}
to find that
$$\pi_\om(\ga)''=\int^\oplus_{\ce(\cs^{\bp_\bn}(\ga))}\pi_\psi(\ga)'' {\rm d}\mu(\psi)\,.$$
The above decomposition enables us to identify the tail algebra. Indeed, by Theorem 4.4.6 in \cite{BR1} and
Proposition \ref{tailtriv} one has
\begin{align*}
\gz_\om^\perp&=\bigcap_{n=1}^\infty\underset{I\in \cf_n}{\bigvee} \pi_\om(\ga(I))''\, =
 \bigcap_{n=1}^\infty\bigvee_{I\in \cf_n}\int^\oplus_{\ce(\cs^{\bp_\bn}(\ga))} \pi_\psi(\ga(I))'' {\rm d}\mu(\psi)\\
&=\int^\oplus_{\ce(\cs^{\bp_\bn}(\ga))}   \bigcap_{n=1}^\infty\bigvee_{I\in \cf_n }      \pi_\psi(\ga(I))'' {\rm d}\mu(\psi)\\
&=\int^\oplus_{\ce(\cs^{\bp_\bn}(\ga))} \bc 1_{\ch_\psi}\di\mu(\psi)=Z_{\bp_\bn}(\om) \,.
\end{align*}
Since by Proposition \ref{large} $\bp_\bn$ acts as a large group of automorphisms on $\ga$, Theorem 3.1 in \cite{St67} applies yielding the existence
of a unique conditional expectation, $E_\om$, from $\pi_\om(\ga)''$ onto $Z_{\bp_\bn}(\om)=\gz^\perp_\om$.

All is left to do is prove the formula for $E_\om$. To this end, note that
$F(X):=\int^\oplus_{\ce(\cs^{\bp_\bn}(\ga))}\langle X_\psi\xi_\psi, \xi_\psi\rangle {\rm d}\mu(\psi)$, $X$ in
$\pi_\om(\ga)''$, defines a conditional expectation of $\pi_\om(\ga)''$ onto $\gz^\perp_\om$ as it is the direct integral
of states. By uniqueness one sees that $F=E_\om$. Finally $E_\om$ is seen to preserve the vector state
$\langle\cdot\xi_\om, \xi_\om\rangle$ by means of simple computations, see also \cite[Theorem 5.3]{CrFid}.
\end{proof}

Before we can state our version of de Finetti's theorem tailored to the present context, we need to recall
 what should be meant by conditional independence  for a net
of local algebras $\{\ga(I): I\in\cp_0(\bn)\}$ with respect to a given state $\om$ of the quasi-local algebra
$\ga$.
We start by recalling that for any such state $\om$ the tail algebra $\gz^\perp_\om$ will always be commutative, see \cite[Theorem 2.6.5]{BR1}, and thus expected.
In other words, there will always exist a conditional expectation $F_\om:\pi_\om(\ga)''\rightarrow \gz^\perp_\om $.
As is customary, we will need to work under the hypothesis that such conditional expectation is normal and $\varphi_{\xi_\om}$-preserving, that is
$\langle F_\om[X]\xi_\om, \xi_\om\rangle=\langle X\xi_\om, \xi_\om \rangle$
for any $X\in\pi_\om(\ga)''$.\\
\begin{defin}
The net $\{\ga(I): I\in\cp_0(\bn)\}$ of the local algebras is conditionally independent with respect to a conditional expectation $F_\om$ as above if
for any $I, J\in\cp_0(\bn)$ with $I\cap J=\emptyset$ we have
$$F_\om[XY]=F_\om[X]F_\om[Y]$$
for every $X\in \pi_\om(\ga(I))''\bigvee\gz^\perp_\om$ and $Y\in \pi_\om(\ga(J))''\bigvee\gz^\perp_\om$;
\end{defin}
We are now ready to state our result.
\begin{prop}\label{deFinetti}
Let $\a$ be a local action of $\bp_\bn$ on a net of local $C^*$-algebras with quasi-local algebra $\ga$.
If $\om\in\cs(\ga)$ is symmetric,  then
$E_\om\circ{\rm ad}_{U_\s^\om}=E_\om$ for every $\s\in\bp_\bn$.
Conversely,  $\om\in\cs(\ga)$ is symmetric if  
$F_\om\circ{\rm ad}_{U_\s^\om}=F_\om$,  $\s\in\bp_\bn$, for some 
normal $\varphi_{\xi_\om}$-preserving
conditional expectation $F_\om: \pi_\om(\ga)''\rightarrow \gz^\perp_\om$.\\
Moreover, in this case the net is conditionally
independent  with respect to $E_\om$.
\end{prop}

\begin{proof}
Suppose that $\om$ is a symmetric  state. Then by Proposition \ref{localtail}, the tail algebra is given by
$\gz^\perp_\om=\int^\oplus_{\ce(\cs^{\bp_\bn}(\ga))}\bc 1_{\ch_\psi} {\rm d}\mu(\psi) $ and the unique
conditional expectation $E_\om: \pi_\om(\ga)''\rightarrow\gz^\perp_\om$ decomposes as
$E_\om(X)=\int^\oplus_{\ce(\cs^{\bp_\bn}(\ga))}\langle X_\psi\xi_\psi, \xi_\psi\rangle {\rm d}\mu(\psi)$  for every $X\in \pi_\om(\ga)''$. We observe that for any $\s\in\bp_\bn$ the unitary $U_\s^\om$ decomposes into a direct integral as well. More precisely, $U_\s^\om=\int^\oplus_{\ce(\cs^{\bp_\bn}(\ga))} U_\s^\psi\di\mu(\psi)$, where $U_\s^\psi$ is the unitary acting on $\ch_\psi$ as $U_\s^\psi \pi_\psi(x)\xi_\psi=\pi_\psi(\alpha_\s(x))\xi_\psi$, $x\in\ga$. Using this decomposition of $U_\s^\om$, it is now straightforward to check that
$E_\om\circ{\rm ad}_{U_\s^\om}=E_\om$, for every $\s\in\bp_\bn$. \\
The converse implication follows by direct computation. Indeed, let
$F_\om: \pi_\om(\ga)''\rightarrow\gz^\perp_\om$ be a conditional expectation 
such that $F_\om\circ{\rm ad}_{U_\s^\om}=F_\om$ for every $\s\in\bp_\bn$.
 For $a\in\ga$ and $\s\in\bp_\bn$ one then has
\begin{align*}
\om(\a_\s(a))=&\langle \pi_\om(\a_\s(a))\xi_\om, \xi_\om \rangle= \langle F_\om\circ{\rm ad}_{U_\s^\om}(\pi_\om(a))\xi_\om, \xi_\om \rangle\\
=&\langle F_\om(\pi_\om(a))\xi_\om, \xi_\om \rangle=\om(a)\,,
\end{align*}
which shows that $\om$ is symmetric, and thus $F_\om=E_\om$
thanks to Proposition \ref{localtail}.

As for conditional independence, fix now $I_1, I_2\subset\bn$ finite subsets with $I_1\cap I_2=\emptyset$, and for $i=1, 2$ take
$X_i\in \pi_\om(\ga(I_i))''\bigvee\gz^\perp_\om$. We need to show that
$E_\om(X_1X_2)= E_\om(X_1)E_\om(X_2)$. To this end, we start by considering
two localized elements, that is $X_i\in \pi_\om(\ga(I_i))$, $i=1, 2$, with $I_1\cap I_2=\emptyset$.
In this case, by using Proposition \ref{localtail} and (3) in the statement of  Theorem \ref{extremelocal} we have
\begin{align*}
&E_\om(X_1 X_2)\\
&=\int^\oplus_{\ce(\cs^{\bp_\bn}(\ga))}\langle X_{1, \psi}X_{2, \psi}\,\xi_\psi, \xi_\psi\rangle 1_{\ch_\psi} {\rm d}\mu(\psi)\\
&=\int^\oplus_{\ce(\cs^{\bp_\bn}(\ga))} \langle X_{1, \psi}\,\xi_\psi, \xi_\psi\rangle  \langle X_{2, \psi}\,\xi_\psi, \xi_\psi\rangle 1_{\ch_\om} \di\mu(\psi)\\
&=\int^\oplus_{\ce(\cs^{\bp_\bn}(\ga))} \langle X_{1, \psi}\,\xi_\psi, \xi_\psi\rangle 1_{\ch_\om}\di\mu(\psi) \int^\oplus_{\ce(\cs^{\bp_\bn}(\ga))}\langle X_{2, \psi}\,\xi_\psi, \xi_\psi\rangle 1_{\ch_\om} \di\mu(\psi)\\
&=E_\om(X_1) E_\om(X_2)\, .
\end{align*}
Since $E_\om$ is a normal conditional expectation, by density the  above equality still holds
for $X_i\in\pi_\om(\ga(I_i))''$, $i=1, 2$.\\
We are now ready to deal with the general case. As $\gz^\perp_\om$ is contained in the center of
$\pi_\om(\ga)''$, we may assume that $X_i$, $i=1, 2$, is of the form
$X_i=\sum_{j\in F} T^i_j  C^i_j$, where $F$ is a finite set, $\{T^i_j: j\in F\}\subset\pi_\om(\ga(I_i))''$ and
$\{C^i_j: j\in F\}\subset\gz^\perp_\om$ for $i=1, 2$. Since
$X_1X_2=\sum_{j, l\in F}T^1_j T^2_l C^1_j C^2_l$, we find
\begin{align*}
E_\om(X_1X_2)&= \sum_{j, l\in F} E_\om(T^1_j T^2_l)C^1_jC^2_l= \sum_{j, l\in F} E_\om(T^1_j) E_\om (T^2_l)C^1_jC^2_l \\
&=\sum_{j\in F}E_\om(T^1_j)C^1_j\sum_{l\in F}E_\om(T^2_l)C^2_l=E_\om(X_1)E_\om(X_2)\,
\end{align*}
and the proof is complete.
\end{proof}

If  more assumptions are made on the structure of the net of the local algebras, the global invariance condition
$E_\om\circ{\rm ad}_{U_\s^\om}=E_\om$ can be recast in a seemingly weaker way.
To this end, from now on we will assume that the net $\{\ga(I): I\in\cp_0(\bn)\}$
is {\it additive}\footnote{The terminology is borrowed from algebraic quantum field theory, see Definition 4.13 in \cite{A}.}, namely that $\ga(I\cup J)= C^*(\ga(I), \ga(J))$  for every $I, J\in\cp_0(\bn)$.
In particular, for any finite subset $I\subset\bn$ we have
$\ga(I)= C^*(\ga(\{i\}): i\in I)$.\\
In  the present context Theorem \ref{deFinetti} can be stated as follows.

\begin{prop}\label{deFinetti2}
Let $\a$ be a local action of $\bp_\bn$ on an additive net of local $C^*$-algebras.
A  state $\om\in\cs(\ga)$ on the quasi-local algebra $\ga$  is symmetric if and only if:
\begin{itemize}
\item [(i)] the local algebras are conditionally independent w.r.t. $E_\om$;
\item [(ii)] for every $i\in\bn$, $E_\om[\pi_\om(\a_\s(a))]=E_\om[\pi_\om(a)], \, a\in\ga(\{i\}), \s\in\bp_\bn$.
\end{itemize}
\end{prop}

\begin{proof}
By virtue of Theorem \ref{deFinetti} we need only show that (i) and (ii) imply
that $E_\om\circ{\rm ad}_{U_\s^\om}=E_\om$.\\
Since $E_\om$ is a normal conditional expectation, by density
of $\pi_\om(\ga)$ in the bicommutant $\pi_\om(\ga)''$, it is enough to verify
the equality on $\pi_\om(\ga)$. As $\ga$ is in turn the inductive limit of the local algebras, the thesis will be achieved if
we show that for any fixed finite subset $I\subset\bn$ one has $E_\om[\a_\s(\pi_\om(a))]=E_\om[\pi_\om(a)]$ for every
$a\in\ga(I)$.
By additivity there is no loss of generality to assume that $a$ factors into a product as $a=a_{i_1}a_{i_2}\cdots a_{i_n}$,
with  $i_l\neq i_k$ (possible repetitions of the same index at different places are dealt with by means of
(iii) in Definition \ref{lnet}). For any $\s\in\bp_\bn$ we then have:
\begin{align*}
&E_\om[ \pi_\om(a_{i_1}a_{i_2}\cdots a_{i_n})]= E_\om[\pi_\om(a_{i_1})]E_\om[\pi_\om(a_{i_2})]\cdots E_\om[\pi_\om(a_{i_n})] \\
&=E_\om[\pi_\om(\a_\s(a_{i_1}))]E_\om[\pi_\om(\a_\s(a_{i_2}))]\cdots E_\om[\pi_\om(\a_\s(a_{i_n}))] \\
&= E_\om[\pi_\om(\a_\s(a_{i_1}a_{i_2}\cdots a_{i_n}) ]\,,
\end{align*}
which ends the proof.
\end{proof}
It is worth noting that when the quasi-local algebra arises as a quotient of the infinite free
product $\ast_\bn \gb$ of a sample $C^*$-algebra $\gb$, the conditions (i) and (ii) in the statement above return
the usual notion for a (quantum) stochastic process to be conditionally independent and identically distributed with respect to
the tail algebra. Indeed, in this case  the states of the quotient are in a one-to-one correspondence with the stochastic processes on the sample algebra
$\gb$, see {\it e.g.} Theorem 3.4 and Definition 4.1 in \cite{CrFid} or Theorem 2.3 in \cite{CrFid2}.

\section{Processes on infinite graded tensor products}\label{infprod}
In this section, we collect some results on $\bz_2$-graded algebraic structures obtained as tensor products of graded $*$-algebras.
Consider the $C^*$-algebras $\ga_1$ and $\ga_2$, and denote by $\ga_1\otimes \ga_2$ the algebraic tensor product $\ga_1\odot \ga_2$ with the product and involution given by
$$
(a_1\otimes a_2)(a_1'\otimes a_2'):=a_1a_1'\otimes a_2a_2'\,,\quad (a_1\otimes a_2)^*:=a_1^*\otimes a_2^*\,,
$$
for all $a_1,a_1'\in\ga_1$, $a_2,a_2'\in\ga_2$. Let us denote by $\ga_1\otimes_{\max} \ga_2$ and $\ga_1\otimes_{\min} \ga_2$ the completions of $\ga_1\otimes \ga_2$ with respect to the maximal and minimal $C^*$-cross norm, respectively, see \cite{T1}.\\
If one takes $\om_1\in\cs(\ga_1)$ and $\om_2\in\cs(\ga_2)$, their product state $\psi_{\om_1,\om_2}\in\cs(\ga_1\otimes_{\min} \ga_2)$ is well defined also on $\ga_1\otimes_{\max} \ga_2$, and consequently the notation $\psi_{\om_1,\om_2}\in\cs(\ga_1\otimes \ga_2)$ will be used in the sequel.\\

Suppose that $(\ga_1,\th_1)$ and $(\ga_2,\th_2)$ are $\bz_2$-graded $*$-algebras, and consider the linear space $\ga_1\odot \ga_2$. In what follows, we recall the definition of the involutive $\bz_2$-graded tensor product, which will be henceforth denoted by
$\ga_1\hat{\otimes} \ga_2$. For homogeneous elements $a_1\in\ga_1$, $a_2\in\ga_2$ and $i,j\in\bz_2$, we set
\begin{eqnarray*}
\begin{split}
\label{ecfps}
&\eps(a_1,a_2):=\left\{\!\!\!\begin{array}{ll}
                      -1 &\text{if}\,\, \partial(a_1)=\partial(a_2)=-1\,,\\
                     \,\,\,\,\,1 &\text{otherwise}\,.
                    \end{array}
                    \right.\\
&\eeps(i,j):=\left\{\!\!\!\begin{array}{ll}
                      -1 &\text{if}\,\, i=j=-1\,,\\
                     \,\,\,\,\,1 &\text{otherwise}\,.
                    \end{array}
                    \right.
\end{split}
\end{eqnarray*}
Given $x,y\in\ga_1\odot\ga_2$ with
\begin{align*}
\begin{split}
\label{xy}
&x:=\oplus_{i,j\in\bz_2}x_{i,j}\in\oplus_{i,j\in\bz_2}(\ga_{1,i}\odot\ga_{2,j})\,,\\
&y:=\oplus_{i,j\in\bz_2}y_{i,j}\in\oplus_{i,j\in\bz_2}(\ga_{1,i}\odot\ga_{2,j})\,,
\end{split}
\end{align*}
 the involution, which by a minor abuse of notation we continue
to denote by $^*$, and  the multiplication on $\ga_1\hat{\otimes} \ga_2$ are defined as (see also \emph{e.g.} \cite{CDF})
\begin{equation*}
\label{prstc}
\begin{split}
x^*:=&\sum_{i,j\in\bz_2}\eeps(i,j)x_{i,j}^*\, ,\\
xy:=&\sum_{i,j,k,l\in\bz_2}\eeps(j,k)x_{i,j}{\bf}y_{k,l}\,.
\end{split}
\end{equation*}

The $*$-algebra thus obtained, in \cite{CDF} referred to as the Fermi tensor product of $\ga_1$ and
$\ga_2$, also carries a $\bz_2$-grading. This 
is induced by the $*$-automorphism $\th=\th_1\hat{\otimes} \th_2$, whose action on  simple tensors is given by
\begin{equation}
\label{aupf}
\th_1\hat{\otimes} \th_2 (a_1\hat{\otimes} a_2):=\th_1(a_1)\hat{\otimes} \th_2(a_2)\,,\quad a_1\in\ga_1\,,\,\, a_2\in\ga_2\,,
\end{equation}
where $a_1\hat{\otimes} a_2$ is nothing but  $a_1\otimes a_2$ thought of
as an element of the $\bz_2$-graded $*$-algebra $\ga_1\hat{\otimes} \ga_2$,
since
$\ga_1\hat{\otimes} \ga_2=\ga_1\otimes \ga_2\,$ as linear spaces.
As of now, we will use $a_1\otimes a_2$ and $a_1\hat{\otimes} a_2$ interchangeably
when no confusion can occur.\\
The even and odd part of the Fermi product are respectively
\begin{equation*}
\label{picm}
\begin{split}
\big(\ga_1\hat{\otimes} \ga_2\big)_+=&\big(\ga_{1,+}\odot\ga_{2,+}\big)\oplus\big(\ga_{1,-}\odot\ga_{2,-}\big)\,,\\
\big(\ga_1\hat{\otimes} \ga_2\big)_-=&\big(\ga_{1,+}\odot\ga_{2,-}\big)\oplus\big(\ga_{1,-}\odot\ga_{2,+}\big)\,.
\end{split}
\end{equation*}

\medskip
The construction of the algebraic Fermi tensor product can of course be performed with an arbitrary
number $n$ of $C^*$-algebras $\ga_i$, $i=1, 2, \ldots, n$. As usual, as a linear space
$\ga_1\hat{\otimes} \ga_2\hat{\otimes}\cdots \hat{\otimes} \ga_n$ is given
by the algebraic tensor product $\ga_1\odot \ga_2\, \odot\,\cdots \odot\, \ga_n$. Product
and involution can be defined by carefully exploiting the associativity
of the usual tensor product. 
The $*$-algebra $\ga_1\hat{\otimes} \ga_2\hat{\otimes}\cdots \hat{\otimes} \ga_n$
can be turned into a $\bz_2$-graded algebra by the grading $\theta^{(n)}:= \theta_1\hat{\otimes} \theta_2\hat{\otimes}\cdots \hat{\otimes} \theta_n$ defined as in \eqref{aupf}.\\

For $\om_i\in\cs(\ga_i)$, $i=1,2$, the state $\psi_{\om_1,\om_2}$ has a counterpart in $\ga_1\hat{\otimes} \ga_2$ by means of the product functional $\om_1\times \om_2$, defined as usual by
$$
\om_1\times \om_2\bigg(\sum_{j=1}^n a_{1,j}\hat{\otimes} a_{2,j}\bigg):=\sum_{j=1}^n \om_1(a_{1,j})\om_2(a_{2,j})\,,
$$
for all $\sum_{j=1}^n a_{1,j} \hat{\otimes} a_{2,j}\in \ga_1\hat{\otimes} \ga_2$. Contrary to the case of a trivial grading, the functional defined above is not necessarily positive, unless at least one between $\om_1$ and $\om_2$ is even, see
\cite[Proposition 7.1]{CDF}. More in general, given $\om_i\in\cs(\ga_i)$, we denote by 
$\om_1\times\om_2\times\cdots\times\om_n$
the linear functional on $\ga_1\hat{\otimes} \ga_2\hat{\otimes}\cdots \hat{\otimes} \ga_n$ defined on simple tensors as
\begin{equation*}
\om_1\times\om_2\times\cdots\times\om_n(a_1\hat{\otimes} a_2\hat{\otimes}\cdots \hat{\otimes} a_n):=\om_1(a_1)\om_2(a_2)\cdots\om_n(a_n)
\end{equation*}
for every $a_i\in\ga_i$. The following proposition is a straightforward generalization of \cite[Proposition 2.6]{CRZ}.
 \begin{prop}
Let $(\ga_i,\th_i)$ be graded $C^*$-algebras, $i=1, 2, \ldots n$. Given
$\om_i\in\cs(\ga_i)$, then their product state
 $\om_1\times\om_2\times\cdots\times\om_n$ is positive  if and only if at least
$n-1$ of them are even. Moreover, $\om_1\times\om_2\times\cdots\times\om_n$  is even
if and only if all states $\om_1, \om_2, \ldots, \om_n$ are even.
\end{prop}
\begin{proof}
A simple induction on $n$.
\end{proof}

As in the case with only two factors,  which has been addressed in \cite{CRZ}, the product $\ga_1\hat{\otimes} \ga_2\hat{\otimes}\cdots \hat{\otimes} \ga_n$ will in general admit many $C^*$-completions.
We first consider the minimal completion, namely the one obtained by completing with respect to the norm
$$\|x\|_{\rm min}:=\sup\{\|\pi_\om(x)\|: \om=\om_1\times\om_2\cdots\times\om_n, \, \om_i\in\mathcal{S}_+(\ga_i), \, i=1, \ldots, n\}\,,$$
which we denote by $\ga_1\hat{\otimes}_{\rm min} \ga_2\hat{\otimes}_{\rm min}\cdots \hat{\otimes}_{\rm min} \ga_n$. 
It is
still a $\bz_2$-graded $C^*$-algebra, with the grading obtained by extending $\theta^{(n)}$ to
the minimal completion, {\it cf.} \cite[Proposition 4.7]{CRZ}.\\

Minimal infinite tensor Fermi  products can  be defined through inductive limits.
More precisely, if $\{(\ga_i, \theta_i): i\in\bn\}$ is a countable family of unital $\bz_2$-graded
$C^*$-algebras, then for each $n\in\bn$ we can consider the injective homomorphism
$$\Phi_n:\ga_1\hat{\otimes}_{\rm min} \ga_2\hat{\otimes}_{\rm min}\cdots \hat{\otimes}_{\rm min} \ga_n\rightarrow \ga_1\hat{\otimes}_{\rm min} \ga_2\hat{\otimes}_{\rm min}\cdots \hat{\otimes}_{\rm min} \ga_n\hat{\otimes}_{\rm min} \ga_{n+1}$$
 completely determined by
$$\Phi_n(a_1\hat{\otimes} a_2\hat{\otimes}\cdots \hat{\otimes} a_n)=
a_1\hat{\otimes} a_2\hat{\otimes}\cdots \hat{\otimes} a_n\hat{\otimes}1\,,$$
for every $a_i\in\ga_i$ and $i=1, \ldots, n$, where by a slight abuse of notation
$1$ denotes the unity of $\ga_i$ for every $i\in\bn$.\\
Clearly, $\{\ga_1\hat{\otimes}_{\rm min} \ga_2\hat{\otimes}_{\rm min}\cdots \hat{\otimes}_{\rm min} \ga_n, \Phi_n\}$ is an inductive system of $C^*$-algebras, whose limit we denote by
$\displaystyle{ \hat{\otimes}_{\rm min}^{i\in\bn}\ga_i}$ and call the (minimal) infinite Fermi tensor product.\\
We denote by $\iota_n$ the embedding of $\ga_1\hat{\otimes}_{\rm min} \ga_2\hat{\otimes}_{\rm min}\cdots \hat{\otimes}_{\rm min} \ga_n$ 
into $\hat{\otimes}_{\rm min}^{i\in \bn}\ga_i$. Henceforth, we will often write $a_1\hat{\otimes} a_2\hat{\otimes}\cdots \hat{\otimes} a_n\hat{\otimes}1\,\hat{\otimes}1\,\cdots $ rather than  write
$\iota_n(a_1\hat{\otimes} a_2\hat{\otimes}\cdots \hat{\otimes} a_n)$, as is commonly done in the
literature.\\

Infinite Fermi tensor products provide examples of quasi-local algebras. Here, the net of local
subalgebras is as follows. For every $I=\{i_1, \ldots, i_{|I|}\}$ finite subset of $\bn$, we denote by
$\ga(I)\subset\displaystyle{ \hat{\otimes}_{\rm min}^{i\in\bn}\ga_i}$
the unital $C^*$-subalgebra generated by simple tensors of the type
$$1\hat{\otimes}\cdots \hat{\otimes}a_{i_1}\hat{\otimes}\cdots\hat{\otimes} a_{i_2}\hat{\otimes}\cdots \hat{\otimes} a_{i_{|I|}}\hat{\otimes}1\hat{\otimes}1\cdots$$
when the $a_{i_j}$'s vary in $\ga_{i_j}$, $j=1, \ldots, |I|$.\\

If now $\ga$ is a fixed unital $C^*$-algebra, we denote by $\ga^{(n)}$ the minimal
Fermi tensor product of $\ga$ with itself $n$ times and by
 $\hat{\otimes}_{\rm min}^{\bn}\ga$ the corresponding infinite graded tensor product.
We still denote by $\iota_n:\ga^{(n)}\rightarrow \hat{\otimes}_{\rm min}^{\bn}\ga$ the embeddings of
$\ga^{(n)}$ into $\hat{\otimes}_{\rm min}^{\bn}\ga$.\\
 Finally, note that  $\hat{\otimes}_{\rm min}^{\bn}\ga$ is still
a $\bz_2$-graded $C^*$-algebra, whose grading, which we denote by  $\hat{\otimes}^{\bn}\,\theta$, is
obtained as the inductive limit ot the $\theta^{(n)}$'s.\\

In the following example we show how the CAR algebra can be re-obtained as an infinite Fermi
tensor product.
\begin{exa}\label{CAR}
Our starting data is the $\bz_2$-graded $C^*$-algebra $(\ga, \theta)=(\bm_2(\bc), {\rm ad}(U))$,
where $U$ is the (Pauli) unitary matrix
$
U=  \left(
\begin{array}{rr}
1 & 0 \\
0 &{ -1}
\end{array}
\right)
$. Note that, given $B \in\bm_2(\bc)$, one has $UBU^*=B$ if and only if
$B$ is a diagonal matrix and $UBU^*=-B$
if and only if $B$ is anti-diagonal.\\
We next show that $\hat{\otimes}^{\bn}\bm_2(\bc)$ is
$*$-isomorphic with the CAR algebra and  $\hat{\otimes}^\bn {\rm ad}(U)$
is its usual grading. To this end, set $A:=\left(
\begin{array}{ll}
0 & 1 \\
0 & 0
\end{array}
\right)$. Note that $A$ is odd and $A^2=0,\, A^*A+AA^*=I$.
For every $j\in\bn$, denote by $i_j: \bm_2(\bc)\rightarrow\hat{\otimes}^{\bn}\bm_2(\bc)$
the injective $^*$-homomorphism given by
$$i_j(B)=1\hat{\otimes} 1\hat{\otimes}\cdots \hat{\otimes}\underbrace{B}_{\textrm{j-th place}}\,\hat{\otimes}1\hat{\otimes}1\cdots,\, B\in\bm_2({\bc}).$$
and define $a_j:=i_j(A)$, $j\in\bn$. Since $A$ generates $\bm_2(\bc)$, $\{a_j: j\in\bn\}$ is a set of generators
of the infinite Fermi tensor product of $\bm_2(\bc)$ with itself.
Now the relations $a_ja_k+a_ka_j=0$ and $a_ja_k^*+a_k^*a_j=\delta_{j, k}I$, $j, k\in\bn$, are a straighforward
consequence of the equalities $A^2=0, \,A^*A+AA^*=I$ and of the fact that $A$ is odd.
As the CAR algebra, ${\rm CAR}(\bn)$, is the universal $C^*$-algebra generated by $b_j$'s satisfying the above relations,
we find that there must exist a surjective  $*$-homomorphism
$\Psi:{\rm CAR}(\bn)\rightarrow \hat{\otimes}^{\bn}\bm_2(\bc)$ such that
$\Psi(b_j)=a_j$, for every $j\in\bn$. By simplicity of ${\rm CAR}(\bn)$, $\Psi$ is also injective and so
${\rm CAR}(\bn)\cong\hat{\otimes}^{\bn}\bm_2(\bc)$. Finally, as for the grading, it is
enough to observe that each $a_j$ is odd w.r.t. $\hat{\otimes}^\bn {\rm ad}(U)$.
\end{exa}

If $\om_i\in\cs_+(\ga_i)$ are even states for every $i\in\bn$, then their infinite product,
$\underset{i\in\bn}{\times}\om_i$, is the state on $\hat{\otimes}^{i\in\bn}_{\rm min}\ga_i$
uniquely determined by
$$
\underset{i\in\bn}{\times}\om_i(x_1\, \hat{\otimes}\cdots \hat{\otimes} x_n\,\hat{\otimes} 1\,\hat{\otimes} 1\,\hat{\otimes}\cdots)=\om_1(x_1)\om_2(x_2)\cdots\om_n(x_n)
$$
for every $x_i\in\ga_i$, $i=1, 2, \ldots, n$, and every $n\in\bn$.\\
If $\om_i$ is a fixed even state $\om$ on $\ga$ for each $i\in\bn$,  we then simply denote
by $\times^\bn\om$ the product state $\underset{i\in\bn}{\times}\om_i$ on $\hat{\otimes}^\bn_{\rm min}\ga$ .\\

\medskip

Similarly to what we have seen for $C^*$-algebras, the definition of the $\bz_2$-graded tensor product of two Hilbert spaces, as given in \cite{CRZ}, can easily be extended to an arbitrary number of spaces.
We start by recalling that a $\bz_2$-graded Hilbert space is a pair $(\ch, U)$, where $\ch$ is a
(complex) Hilbert space and $U$ a self-adjoint unitary acting on $\ch$. 
In such a situation, $\ch$ decomposes into an orthogonal direct sum of the type
$$\ch=\ch_+\oplus\ch_-\, ,$$
where $\ch_+:= {\rm Ker}(I-U)$ and $\ch_-:= {\rm Ker}(I+U)$.
As usual, vectors in $\ch_+$ ($\ch_-$) are called {\it even} ({\it odd}) vectors.
Even or odd vectors are collectively referred to as homogeneous vectors.\\
The Hilbert tensor product $\ch_1\otimes\ch_2$ of two
 $\bz_2$-graded Hilbert spaces $(\ch_1, U_1)$ and $(\ch_2, U_2)$
will always be conceived of as a graded Hilbert space, with the natural grading
associated with $U_1\otimes U_2$.

We also recall that infinite  tensor products of Hilbert spaces can be defined as direct limits of finite products
 following a construction due to von Neumann, which we rather quickly sketch for convenience.
We first observe that, given two  Hilbert spaces $\ch_1$ and $\ch_2$, for any
unit vector $\xi\in \ch_2$  the map $\ch_1\ni x\rightarrow x{\otimes} \xi\in\ch_1\,{\otimes} \ch_2$ is isometric. Given a sequence $\{(\ch_i, \xi_i): i\in\bn\}$ of Hilbert spaces, where for each $i\in\bn$ $\xi_i\in\ch_i$ is a unit vector, we can consider the isometries $\Phi_n: \ch_1{\otimes} \ch_2{\otimes}\cdots {\otimes} \ch_n\rightarrow \ch_1{\otimes} \ch_2{\otimes}\cdots {\otimes} \ch_{n+1}$ given by
$$\Phi_n(x_1{\otimes} x_2{\otimes}\cdots {\otimes} x_n):=x_1{\otimes} x_2{\otimes}\cdots{\otimes} x_n\,{\otimes}\xi_{n+1}.$$
The infinite tensor product of the Hilbert spaces $\ch_i$ with respect to the sequence $\boldsymbol{\xi}=\{\xi_i\}_{i\in\bn}$ is by definition
the inductive limit of the direct system $\{(\ch_1{\otimes} \ch_2{\otimes}\cdots {\otimes} \ch_n, \Phi_n): n\in\bn\}$, and will be denoted
by $\underset{\boldsymbol{\xi}}{{\otimes}}\, \ch_i$.
For each integer $n$, we denote by $\iota_n$ the isometric embedding of $\ch_1{\otimes} \ch_2{\otimes}\cdots {\otimes} \ch_n$ into $\underset{\boldsymbol{\xi}}{{\otimes}}\, \ch_i$. Note that $\iota_{n+1}\circ\Phi_n=\iota_n$ for every $n$ by definition
of inductive limit. Instead of $\iota_n(x_1{\otimes} x_2{\otimes}\cdots {\otimes} x_n)$ 
we will every so often write $x_1{\otimes} x_2{\otimes}\cdots {\otimes} x_n{\otimes}\xi_{n+1}{\otimes}\xi_{n+2}{\otimes}\cdots$. 
When all the Hilbert spaces $\ch_i$ are graded, say by self-adjoint unitaries $U_i\in\cb(\ch_i)$, then the infinite product
$\underset{\boldsymbol{\xi}}{{\otimes}}\, \ch_i$ can be equipped with
a $\bz_2$-grading through the self-adjoint unitary ${\otimes}_{n\in\bn}U_n$, whose definition is deferred to  Section \ref{twistedcomm}.\\

In \cite{CRZ} the Fermi product of
grading-equivariant representations was defined for two representations. Obviously, the construction
given there continues to work with an arbitrary number $n$ of representations.
In other terms, if $\pi_i: \ga_i\rightarrow \cb(\ch_i)$, with $i=1, 2, \ldots, n$, are grading-equivariant
representations acting on the $\bz_2$-graded Hilbert spaces $(\ch_i, U_i)$, then it is possible to define a representation
$\pi$ of the Fermi product $\ga_1\hat{\otimes} \ga_2\hat{\otimes}\cdots \hat{\otimes} \ga_n$ acting on the Hilbert space $\ch:= \ch_1{\otimes} \ch_2{\otimes}\cdots{\otimes} \ch_n$ as
$$\pi(a_1\hat{\otimes} a_2\hat{\otimes}\cdots \hat{\otimes} a_n):=
\pi_1(a_1)\hat{\otimes} \pi_2(a_2)\hat{\otimes}\cdots \hat{\otimes} \pi_n(a_n)$$
for every $a_i\in\ga_i$, $i=1, 2, \ldots, n$. Note that in the formula above
the symbol $\hat{\otimes}$ is actually used to denote the Fermi tensor product of two or more
operators acting on (possibly different) Hilbert spaces, as defined in \cite{CRZ}.
For convenience, we recall the definition with two (homogeneous) operators $T_i\in\cb(\ch_i)$, $i=1, 2$:
$$T_1\hat{\otimes} T_2 (\xi_1\otimes\xi_2):=\ T_1\xi_1\otimes T_2\xi_2$$
for homogeneous vectors $\xi_i\in\ch_i$, where, as usual, the sign $\varepsilon(T_2, \xi_1)$
is $-1$ if $T_2$ and $\xi_1$ are both odd.  Notice that if $T_1$ and $T_2$ are both even, then
$T_1\hat{\otimes} T_2$ is nothing but the usual tensor product $T_1\otimes T_2$
of operators.\\
We will say that $\pi$ is the Fermi tensor product of the representations
$\pi_i$ and write $\pi=\pi_1\hat{\otimes} \pi_2\hat{\otimes}\cdots \hat{\otimes} \pi_n$. 
It is easy to verify that $\pi$ can be extended to a representation of the
completion $\ga_1\hat{\otimes}_{\rm min} \ga_2\hat{\otimes}_{\rm min}\cdots \hat{\otimes}_{\rm min} \ga_n$.
With a slight abuse of notation, we continue to denote  this extension by $\pi_1\hat{\otimes} \pi_2\hat{\otimes}\cdots \hat{\otimes} \pi_n$.
Once finitely many representations have been dealt with, infinitely many
representations can be handled easily. For the sake of simplicity, we only consider cyclic representations.\\
Indeed, if $\pi_{\om_i}: \ga_i\rightarrow \cb(\ch_{\om_i})$, with $i\in\bn$, are the
GNS representations of the even states $\om_i\in\cs_+(\ga_i)$, then their Fermi product
is the representation $\underset{i\in\bn}{\hat{\otimes}}\, \pi_{\om_i}$ of the
Fermi product $C^*$-algebra ${\hat{\otimes}}_{\rm min}\ga_i$ acting on
the Hilbert space $\underset{\boldsymbol{\xi}}{{\otimes}}\, \ch_{\om_i}$, with
$\boldsymbol{\xi}:=\{\xi_{\om_i}\}_{i\in\bn}$, uniquely determined by
$$
\underset{i\in\bn}{\hat{\otimes}}\, \pi_{\om_i}(a_1\, \hat{\otimes}\cdots \hat{\otimes} a_n\,\hat{\otimes} 1\,\hat{\otimes} 1\,\hat{\otimes}\cdots )=
\pi_{\om_1}(a_1)\hat{\otimes}\cdots\hat{\otimes} \pi_{\om_n}(a_n)\,\hat{\otimes} 1\,\hat{\otimes} 1\,\hat{\otimes}\cdots.
$$

\begin{rem}\label{prodrep}
The representation $\underset{i\in\bn}{\hat{\otimes}}\, \pi_{\om_i}$ is still cyclic
and coincides (up to unitary equivalence) with the GNS representation of
the product state $\om=\underset{i\in\bn}{\times}\om_i$. Indeed, the unit vector $\xi$
defined  as 
$$\xi:=\iota_n(\xi_{\om_1}\, {\otimes}\cdots {\otimes} \xi_{\om_n})=:\underset{i\in\bn}{{{\otimes}}}\, \xi_{\om_i}$$ 
(note that the definition does not depend on $n$)
is cyclic, and the equality $\om(\cdot)= \langle\underset{i\in\bn}{{\hat{\otimes}}}\, \pi_{\om_i}(\cdot) \xi, \xi\rangle$ is easily checked.
\end{rem}

\begin{rem}\label{faithful}
Note that if we start with a faithful even state $\om$ on $\ga$, then its product
$\times^\bn\om$ will  be a faithful state on  $\hat{\otimes}_{\rm min}^{\bn}\ga$. This can  actually be seen much in the
same way as in the proof of Proposition 5  in \cite{Gui2}, page 22. 
\end{rem}

Finally, we will simply denote by ${{\otimes}}^\bn\, \ch_\om$ the infinite  product
$\underset{\boldsymbol{\xi}}{{\otimes}}\, \ch_{\om_i}$, with
$\ch_{\om_i}=\ch_\om$ and $\boldsymbol{\xi}$ being 
sequence constantly equal to $\xi_\om$.\\

As already remarked, the infinite Fermi tensor product $\hat{\otimes}_{\rm min}^{\bn}\ga$ of a given
$C^*$-algebra $\ga$ is a quasi-local
algebra. In addition, $\hat{\otimes}_{\rm min}^{\bn}\ga$ is acted upon by $\bp_\bn$ in a natural way.
Indeed, associated with any $\s\in\bp_\bn$ there is a $*$-automorphism
$\a_\s\in {\rm Aut}(\hat{\otimes}_{\rm min}^{\bn}\ga)$, which is completely
determined by
$$\a_\s(\iota_n(a_1\hat{\otimes} a_2\hat{\otimes}\cdots \hat{\otimes} a_n))=\iota_n(a_{\s(1)}\hat{\otimes} a_{\s(2)}\hat{\otimes}\cdots \hat{\otimes} a_{\s(n)})\,,$$
for $n\in\bn$, $a_i\in\ga$, $i=1, 2, \ldots, n$.  Clearly, this is a local action of $\bp_\bn$ in the sense of
Definition \ref{locact}.
In particular, 
Theorem \ref{extremelocal} applies to the present context 
in a strengthened fashion. More precisely, the extreme symmetric states can now be characterized as infinite products
of a given state on $\ga$.

\begin{prop}\label{extremeprod}
Let $(\ga, \theta)$ be a $\bz_2$-graded $C^*$-algebra. If $\om$ is a symmetric state on  $\hat{\otimes}_{\rm min}^{\bn}\ga$, then the following are equivalent:
\begin{enumerate}
\item $\om$ is extreme;
\item $\om $ is strongly clustering;
\item there exists an even state $\r\in\cs(\ga)$ such that $\om=\times^\bn \r$.
\end{enumerate}
\end{prop}

\begin{proof}
In light of Theorem \ref{extremelocal}  we need only show that (2)
and (3) are equivalent, which can be done exactly as in \cite[Theorem 5.3]{CFCMP}.
\end{proof}

As a consequence of the above result, we find that the extreme symmetric states of $\hat{\otimes}_{\min}^{\bn}\ga$ are 
sufficiently many to separate its points. This circumstance plays an
instrumental role in proving weak ergodicity of the permutation action on an infinite product, as shown below.

\begin{prop}\label{weakerg}
For any given $\bz_2$-graded $C^*$-algebra $(\ga, \theta)$, 
the $C^*$-dynamical system  $(\hat{\otimes}_{\rm min}^{\bn}\ga, \bp_\bn, \{\a_\s: \s\in\bp_\bn\})$ is weakly ergodic, {\it i.e.}
 $x\in\hat{\otimes}_{\rm min}^{\bn}\ga$ with $\a_\s(x)=x$ for every $\s\in\bp_\bn$ implies $x=\lambda 1$
for some $\lambda\in \bc$.
\end{prop}

\begin{proof}

We start by showing that for any given $x$ in $\hat{\otimes}_{\rm min}^{\bn}\ga$ there exists a separable $\bz_2$-graded subalgebra $\widetilde{\ga}\subset\ga$ such that
$x$ belongs to $\hat{\otimes}_{\rm min}^{\bn}\widetilde{\ga}$.
Indeed, by definition $x$ is the limit in norm of a sequence
$\{x_n\}_{n\in\bn}$, where the $x_n$'s
are elements of the form
$$x_n=\sum_{k\leq K_n} a_{1, k}^{(n)}\,\hat{\otimes} a_{2, k}^{(n)}\, \hat{\otimes}\cdots\hat{\otimes} a_{L_n, k}^{(n)}\,\hat{\otimes}1\hat{\otimes} 1\cdots$$
where $K_n, L_n$ are suitable integers and $a_{i, k}^{(n)}$ belongs to $\ga$ for every integers $k\leq K_n$ and $i\leq L_n$.
The countably generated $C^*$-subalgebra
$$\widetilde{\ga}:=C^*\{a_{i, k}^{(n)}, \theta(a_{i, k}^{(n)}): i\leq L_n, k\leq K_n, n\in\bn\}\subset\ga$$
clearly does the job.
Let now $x$ in $\hat{\otimes}_{\rm min}^{\bn}\ga$ be a fixed point, that is
$\a_\s(x)=x$ for every $\s\in\bp_\bn$, and let
$\rho$ be a faithful even state on the separable $C^*$-algebra $\widetilde{\ga}$
considered above. Note that the action of
$\bp_\bn$ leaves $\hat{\otimes}_{\rm min}^{\bn}\widetilde{\ga}$
invariant. Define $\om$ as the infinite
product of $\r$ with itself.  By Remark \ref{faithful} $\om$ is still faithful.
Since $x$ is invariant under the action of all
$\s$'s in $\bp_\bn$, we have that
$\pi_\om(x)\xi_\om$ lies in $\ch_\om^{\bp_\bn}$. By virtue of
Propositions \ref{symmquasi} and \ref{extremeprod}
there exists $\lambda\in\bc$ such that
$\pi_\om(x)\xi_\om=\lambda\xi_\om$.
By faithfulness of $\om$ we find $x=\lambda 1$.
\end{proof}

The compact convex set $\cs^{\bp_\bn}(\hat{\otimes}_{\rm min}^{\bn}\ga)$ is again
a Choquet simplex. Moreover, its extreme points make up a closed set since the bijection
$\cs_+(\ga)\ni\r\overset{T}{\mapsto} \times^\bn\r\in \ce(\cs^{\bp_\bn}(\hat{\otimes}_{\rm min}^{\bn}\ga))$
establishes a homeomorphism between the two topological spaces.
In particular, for any $\om\in\cs^{\bp_\bn}(\hat{\otimes}_{\rm min}^{\bn}\ga) $, there exists
a unique probability measure $\mu$ which is now genuinely supported on $\ce(\cs^{\bp_\bn}(\hat{\otimes}_{\rm min}^{\bn}\ga))$
such that
\begin{equation*}\label{Choq1}
\om=\int_{\ce(\cs^{\bp_\bn}(\hat{\otimes}_{\rm min}^{\bn}\ga))} \psi\, {\rm d}\mu(\psi)\,.
\end{equation*}
 Because $\cs_+(\ga)$ and $\ce(\cs^{\bp_\bn}(\hat{\otimes}_{\rm min}^{\bn}\ga))$ are homeomorphic compact spaces, the above equality can also be rewritten
as
\begin{equation*}\label{Choq2}
\om=\int_{\cs_+(\ga)} \times^\bn\r\,{\rm d}\mu^*(\r)\, ,
\end{equation*}
where $\mu^*$ is the probability  measure on  $\cs_+(\ga)$ induced by $\mu$ through $T$, {\it i.e.}
$\mu^*(B)= \mu(T(B))$, for any Borel set $B\subset\cs_+(\ga)$.\\

The structure of our Choquet simplex can be further analyzed by
spotting its faces. This was done in \cite{St} in the case of usual infinite tensor products and can be easily
adapted to the present situation. More explicitly, Theorem 2.9 and Corollary 2.10 in \cite{St} admit a straightforward
extension to the graded case. In order to the state it, we keep the same notation as in the above mentioned paper
and denote by $X$ the generic type of a given von Neumann algebra. In other words, $X$ can be  $I, II_1 , II_\infty,  III$. 
\begin{prop}
 For any $X$ as above, define the convex subsets
$$\cs^{\bp_\bn}(\hat{\otimes}_{\rm min}^{\bn}\ga)_X:=\{\om\in\cs^{\bp_\bn}(\hat{\otimes}_{\rm min}^{\bn}\ga): \pi_\om(\hat{\otimes}_{\rm min}^{\bn}\ga)''\, \textrm{is of type}\, X\}\,.$$
Then $\cs^{\bp_\bn}(\hat{\otimes}_{\rm min}^{\bn}\ga)_X$ is a face of $\cs^{\bp_\bn}(\hat{\otimes}_{\rm min}^{\bn}\ga)$.
Moreover, $\cs^{\bp_\bn}(\hat{\otimes}_{\rm min}^{\bn}\ga)$
is the closed convex hull of the faces $\cs^{\bp_\bn}(\hat{\otimes}_{\rm min}^{\bn}\ga)_X$.
\end{prop}

\medskip

We are next going to draw our attention to maximal completions 
of infinite graded tensor products. As we will recall, these can also
be obtained as quotients of a universal $C^*$-algebra, the infinite free product
of a given $C^*$-algebra. It is this very property that makes maximal completions
particularly suited to establishing a correspondence between their symmetric states
and quantum stochastic processes of a particular form. 
Notably, this allows us to come to a version of de Finetti's
theorem for the processes thus obtained.
With this in mind, we start by quickly outlining how
maximal completions of infinite products can be got to.\\
For finitely many factors $\ga_i$, $i=1, 2, \ldots, n$, the maximal
Fermi tensor product $\ga_1\hat{\otimes}_{\rm max}\ga_2\hat{\otimes}_{\rm max}\cdots\hat{\otimes}_{\rm max}\ga_n$ 
is nothing but the completion of the algebraic product $\ga_1\hat{\otimes}\ga_2\hat{\otimes}\cdots\hat{\otimes}\ga_n$ 
with respect to the maximal $C^*$-norm, see \cite{CDF} for the details.
Infinite products, as usual, are dealt with by taking inductive limits.
Henceforth we will be focusing on the maximal infinite tensor product $\hat{\otimes}_{\rm max}^\bn\ga$ of a given
$\bz_2$-graded $C^*$-algebra $(\ga, \theta)$.\\
First, we observe that $\hat{\otimes}_{\rm max}^\bn\ga$ can also be recovered as a suitable
quotient of the infinite free product $\ast^\bn\ga$ of $\ga$ with itself, see \cite{Av} for
a thorough account of free products.
Note that $\ast^\bn\ga$ is a $\bz_2$-graded $C^*$-algebra with grading
given by $\theta^*:=\ast^\bn\theta$. For every $j\in\bn$, we will denote by
$i_j:\ga\rightarrow\ast^\bn\ga$ the $j$-th embedding of $\ga$ into its infinite free
product, {\it cf.} \cite{Av}. Consider now the closed two-sided ideal
$I$ of $\ast^\bn\ga$ generated by elements of the form $[i_j(a), i_k(b)]_{\theta^*}$
as $a, b$ vary in $\ga$ and $j\neq k$, where for homogeneous $x,y$ in $\ast^\bn\ga$ the symbol $[x,y]_{\theta^*}$ is
the commutator of $x$ and $y$ if at least one of them is even, or the anti-commutator 
when $x$ and $y$ are both odd.\\
An  easy application of \cite[Theorem 8.4]{CDF} shows that the quotient $C^*$-algebra $\ast^\bn\ga/I$ is $*$-isomorphic with
 $\hat{\otimes}_{\rm max}^\bn\ga$.
We will denote by $\Psi:\ast^\bn\ga\rightarrow \hat{\otimes}_{\rm max}^\bn\ga$
the canonical projection onto the quotient.\\
Following \cite{CrFid}, by a quantum stochastic process we mean a quadruple $(\ga, \{\iota_j: j\in\bn\}, \ch, \xi)$, where
$\ga$ is a unital $C^*$-algebra, $\ch$ is a Hilbert space, $\iota_j:\ga\rightarrow \cb(\ch)$ is a
$*$-representation for every $j\in\bn$, and $\xi\in\ch$ is a cyclic vector for the von
Neumann algebra $\bigvee_{j\in\bn}\iota_j(\ga)$.
As shown in \cite{CrFid}, there is a one-to-one correspondence between
stochastic processes on $\ga$ and states of the infinite free
product $\ast^\bn\ga$. This is realized as follows. Starting from a state $\om$
on  $\ast^\bn\ga$, the corresponding process is obtained as 
\begin{equation}\label{one-to-one}
\iota_j:=\pi_\om\circ i_j,\quad j\in\bn.
\end{equation}
Note that the GNS vector $\xi_\om$ is certainly cyclic for $\bigvee_{j\in\bn}\iota_j(\ga)$.
Now $\bp_\bn$ acts naturally on the infinite free product 
$\ast^\bn\ga$. Indeed for any $\s\in\bp_\bn$ there is a unique
automorphism $\a_\s$ of $\ast^\bn\ga$ determined by
$$\a_\s(i_{j_1}(a_1)i_{j_2}(a_2)\cdots i_{j_n}(a_n))= i_{\s(j_1)}(a_1)i_{\s(j_2)}(a_2)\cdots i_{\s(j_n)}(a_n)$$
for $j_1\neq j_2\neq \cdots\neq j_n\in\bn$, $a_1, a_2, \ldots, a_n\in\ga$, $n\in\bn$.
Invariant states under this action of $\bp_\bn$ are again referred to as symmetric states and correspond to so-called
exchangeable processes. We recall that a process $(\ga, \{\iota_j: j\in\bn\}, \ch, \xi)$
is said to be exchangeable if for every $j_1\neq j_2\neq \cdots\neq j_n\in\bn$, $n\in\bn$, $a_1, a_2, \ldots, a_n\in\ga$, and
$\s\in\bp_\bn$ one has
$$\langle \iota_{j_1}(a_1)\iota_{j_2}(a_2)\cdots \iota_{j_n}(a_n)\xi, \xi\rangle= \langle \iota_{\s(j_1)}(a_1)\iota_{\s(j_2)}(a_2)\cdots \iota_{\s(j_n)}(a_n)\xi, \xi\rangle\,.$$

\medskip
We are actually interested in processes $(\ga, \{\iota_j: j\in\bn\}, \ch, \xi)$ where the sample
algebra $\ga$ is in fact a $\bz_2$-graded $C^*$-algebra, and such that for homogeneous $a, b\in\ga$ and $j\neq k$
$\iota_j(a)$ and $\iota_k(b)$ commute if at least one between $a$ or $b$ is even and anti-commute otherwise.
Clearly, processes of this type arise from states of the quotient $\ast^\bn\ga/I\cong\hat{\otimes}_{\rm max}^\bn\ga$, and therefore they will be referred to as $\bz_2$-graded processes on the sample
algebra $\ga$.
Like minimal graded infinite products, maximal ones are seen at once to be quasi-local algebras. In addition,
the natural action of $\bp_\bn$ on them is of course local. As a consequence,
Proposition \ref{localtail} applies, so if $\om$ is a symmetric state on $\hat{\otimes}_{\rm max}^\bn\ga$,
we denote by $E_\om: \pi_\om(\hat{\otimes}_{\rm max}^\bn\ga)''\rightarrow\gz^\perp_\om$ the unique conditional
expectation onto the (commutative) tail algebra.
 That said, we are now ready to state a de Finetti-type theorem for graded processes.

\begin{thm}\label{ffDeFinetti}
A  $\bz_2$-graded process $(\ga, \{\iota_j: j\in\bn\}, \ch, \xi)$, with corresponding
$\om\in\cs(\hat{\otimes}_{\rm max}^\bn\ga)$, is exchangeable if and only if:
\begin{itemize}
\item [(i)] the process is conditionally independent w.r.t. $E_\om$, namely
$$E_\om[XY]=E_\om[X]E_\om[Y]$$
for every $X\in \big(\bigvee_{i\in I} \iota_i(\ga)\big) \bigvee\gz^\perp_\om$ and $Y\in \big(\bigvee_{j\in J} \iota_j(\ga)\big) \bigvee\gz^\perp_\om$, and $I, J\subset \bn$ finite disjoint subsets;
\item [(ii)] the process is identically distributed w.r.t. $E_\om$, namely
$$E_\om[\iota_j(a)]=E_\om[\iota_k(a)]$$ for every $j, k\in\bn$ and $a\in\ga$.
\end{itemize}
\end{thm}

\begin{proof}
It is an application of Proposition \ref{deFinetti2}. Indeed, $\hat{\otimes}_{\rm max}^\bn\ga$
is a quasi-local $C^*$-algebra coming from the additive net of local algebras
$\{\ga(I): I\in\cp_0(\bn)\}$, where $\ga(I)$
is the unital $C^*$-subalgebra generated by simple tensors of the type
$$1\hat{\otimes}\cdots \hat{\otimes}a_{i_1}\hat{\otimes}\cdots\hat{\otimes} a_{i_2}\hat{\otimes}\cdots \hat{\otimes} a_{i_{|I|}}\hat{\otimes}1\hat{\otimes}1\cdots$$
when the $a_{i_j}$'s vary in $\ga$, $j=1, \ldots, |I|$.\\
In order to apply the aforementioned proposition, though, we first need to ascertain that the equality
$\bigvee_{i\in I} \iota_i(\ga)=\pi_\om(\ga(I))''$ holds for any finite subset $I$. This follows by
additivity and \eqref{one-to-one}, for we have
$$
\bigvee_{j\in I} \iota_j(\ga)=\bigvee_{j\in I}\pi_\om(i_j (\ga))=\pi_\om\bigg(C^*(i_j(\ga):j\in I )\bigg)''=\pi_\om(\ga(I))''\,,
$$
where, by a slight abuse of notation, $i_j: \ga\rightarrow\hat{\otimes}_{\rm max}^\bn\ga$ denotes
the map $\Psi\circ i_j$.
\end{proof}

\section{The twisted commutant of a Fermi product and product states}\label{twistedcomm}

The main goal of this section is to prove that an infinite product of even factorial states is still factorial.
This task will be accomplished by making use of the so-called twisted commutant, see \cite{CDF} and the references
therein. For the reader's convenience, though, we recall some basic definitions.
By a $\bz_2$-graded von Neumann algebra we mean a pair $(\mathcal{M}, U)$, where
$\mathcal{M}\subset \cb(\ch)$ is a von Neumann algebra and $U\in \cu(\ch)$ is a self-adjoint
unitary such that $U\mathcal{M}U= \mathcal{M}$.
With such a $U$ it is possible to associate a $*$-automorphism of $(\cb(\ch), {\rm ad}_U)$, commonly
known as \emph{twist automorphism}, see {\it e.g.} \cite{CDF} and references therein, which is defined as
$$\eta_U(T_+ +T_-):= T_++iUT_-$$
for $T=T_+ + T_-$ in $\cb(\ch)$.
The \emph{twisted commutant} of $\mathcal{M}$ is 
$\mathcal{M}^\wr:= \eta_U(\mathcal{M}')=\eta_U(\mathcal{M})'$.
Obviously, the definition makes sense with any 
subset of $\cb(\ch)$.
Again, more details are found in 
\cite{CDF}. Here, we will limit ourselves to observing
that $\eta_U^2={\rm ad}_U$.
We start with a preliminary lemma.

\begin{lem}\label{cyclic}
Let $(\ch, U)$ be a $\bz_2$-graded Hilbert space and let $\ca\subset\cb(\ch)$ be a $*$-algebra
such that $U\ca U=\ca$. A vector $\xi\in\ch$ with $U\xi=\xi$ is cyclic for $\ca$
if and only if it is cyclic for $\eta_U(\ca)$.
\end{lem}

\begin{proof}
We start by observing that if $T\in\ca$, then both $T_+:= \frac{T+UTU}{2}$ and
 $T_-:= \frac{T-UTU}{2}$ are still in $\ca$. The thesis is reached thanks to the following computation
$$\eta_U(T)\xi= T_+\xi + iUT_-\xi= (T_+ -iT_-)\xi$$
which shows that the map $\ca\xi\ni T\xi\mapsto\eta_U(T)\xi\in\ca\xi$ is a linear bijection.
\end{proof}
Here follows a twisted version of Theorem $2$ in \cite{RvD}.
We denote by $\ca_s$ the set of all self-adjoint elements of a given
$*$-algebra $\ca$. Following \cite{RvD}, for any
subspace $K\subset\ch$ we denote by $K^\perp$  the {\it real}
orthogonal complement, namely
$K^\perp=\{x\in\ch: \Re\langle x,\, k \rangle=0, k\in K\}$, 
where $\Re$ is the real part of a complex number.

\begin{lem}\label{Rieffel}
Let $(\ch, U)$ be a $\bz_2$-graded Hilbert space and let
$\ca, \cb\subset \cb(\ch)$ be unital $*$-subalgebras such that $U\ca U=\ca$ and and $U\cb U=\cb$ .
If $\xi\in\ch$ is an even cyclic vector for $\ca$
and $\ca\subset\cb^\wr$, then the following conditions are equivalent:
\begin{enumerate}
\item $\ca^\wr=\cb^{\wr\wr}$;
\item $\eta_U(\ca_s)\xi+i\cb_s\xi$ is dense in $\ch$;
\item $[\eta_U(\ca_s)\xi]^\perp= i \overline{\cb_s\xi}$.
\end{enumerate}
\end{lem}

\begin{proof}
Throughout the proof $\eta_U$ will be simply written as $\eta$ to ease the notation.
We start by showing that $(2)$ and $(3)$ are equivalent.
First, observe $i\cb_s\xi\subset[\eta(\ca_s)\xi]^\perp$. Indeed,
for $B\in \cb_s$ and $A\in\ca_s$, we have that $\eta(A)B$ is self-adjoint
because $\eta(A)$ and $B$ commute since $\ca\subset \cb^\wr$ (that is $\eta(\ca)\subset \cb'$), but then
$\Re\langle iB\xi, \eta(A)\xi \rangle=\Re \langle i\eta(A)B\xi, \xi \rangle=0$.
From the observation above $(2)$
and $(3)$ are seen to be equivalent by a straightforward application of the following
general fact: $X+X^\perp$ is dense in $\ch$ for any real subspace $X\subset\ch$.

We next show that either $(2)$ or $(3)$ implies $(1)$.  First, note that $(2)$ or $(3)$ implies
$(\ca^\wr)_s\xi\subset \overline{\cb_s\xi}$. Indeed, the same computation as above shows that
in general $(\ca^\wr)_s\xi\subset[i\eta(\ca_s)]^\perp$.
Obviously, we only have to prove the inclusion
$\ca^\wr\subset\cb^{\wr\wr}=\cb''$. To this aim, fix $T$ in $(\ca^\wr)_s$ and $R\in \cb'_s$.
We need to show that $RT=TR$. Since $\xi$ is cyclic for $\ca$, by Lemma \ref{cyclic} it is also
cyclic for $\eta(\ca)$, which means it suffices to verify that
$$\langle RTA\xi, C\xi \rangle=\langle TRA\xi, C\xi \rangle $$
for every $A, C\in\eta(\ca)$.
Now there exists a sequence $\{B_n\}_{n\in\bn}\subset \cb_s$ such that
$\|T\xi-B_n\xi\|\rightarrow 0$, and we have
$$\langle RTA\xi, C\xi \rangle=\langle RAT\xi, C\xi \rangle=\lim_n \langle RAB_n\xi, C\xi \rangle=\lim_n
\langle B_nRA\xi, C\xi\rangle\,$$
where in the last equality we have used that by hypothesis the inclusion
$\eta(\ca)\subset\cb'$ holds. Then
\begin{align*}
\lim_n
\langle B_nRA\xi, C\xi\rangle&=\lim_n
\langle RA\xi, B_nC\xi\rangle=\lim_n
\langle RA\xi, CB_n\xi\rangle\\
&=\langle RA\xi, CT\xi\rangle=\langle RA\xi, TC\xi\rangle=\langle TRA\xi, C\xi\rangle\,,
\end{align*}
and we are done.

That $(1)$ implies $(2)$ can be seen in the exact same way as in the proof of Theorem 2
in \cite{RvD} provided that $\ca$ is replaced with $\eta(\ca)$.
\end{proof}

\begin{rem}\label{realorthogonal}
Taking $\ca=\cb^\wr$ in the above result, one finds that $\ca_s\xi+i(\ca^\wr)_s\xi$ is a dense subspace
of $\ch$ and  $[\eta_U(\ca_s)\xi]^\perp= i \overline{(\ca^\wr)_s\xi}$.
\end{rem}

Our aim now is to use Lemma \ref{Rieffel} to come to a twisted version of the tensor
product commutation theorem. For completeness' sake, we recall that this states that the commutant
of the tensor product of two (or infinitely many) von Neumann algebras equals the tensor product of their commutants.
The first general proof was obtained  in \cite{T} and later simplified in \cite{RvD}.\\

We first need to introduce  graded (or Fermi) products of von Neumann algebras. We directly discuss infinite products.
If $\{(\mathcal{M}_n, U_n): n\in\bn\}$ is a family of $\bz_2$-graded
von Neumann algebras on the Hilbert spaces $\ch_n$ and 
$\boldsymbol{\xi}:=\{\xi_n: n\in \bn\}$ is a sequence of unit vectors
$\xi_n\in\ch_n$ such that $U_n\xi_n=\xi_n$ for every $n\in\bn$, the infinite graded product $\underset{\boldsymbol{\xi}}{\hat{\otimes}}\mathcal{M}_n$ is the 
von Neumann algebra on the Hilbert space $\underset{\boldsymbol{\xi}}{{\otimes}}\ch_n$ generated by operators
$T_1\hat{\otimes}T_2\hat{\otimes}\cdots\hat{\otimes}T_k\hat{\otimes}1\hat{\otimes}1\cdots$
with $T_i\in\mathcal{M}_i$ for $i=1, 2, \ldots, k$ and $k\in\bn$.\\
The condition $U_n\xi_n=\xi_n$, $n\in\bn$, comes in useful to define a self-adjoint unitary on $\underset{\boldsymbol{\xi}}{{\otimes}}\ch_n$  as the infinite product ${\otimes}_{n\in\bn}U_n$. This is
understood as the strong limit of the sequence given by finite products of the type
$${\otimes}_{i=1}^nU_i{\otimes}1{\otimes}1{\otimes}\cdots$$
which is easily verified to be Cauchy in the strong operator topology.
The operator ${\otimes}_{n\in\bn}U_n$ thus obtained is a self-adjoint unitary
as it is the limit of self-adjoint unitaries.  Moreover, $\underset{\boldsymbol{\xi}}{\hat{\otimes}}\mathcal{M}_n$ is invariant
under the adjoint action of  ${\otimes}_{n\in\bn}U_n$. Phrased differently, $(\underset{\boldsymbol{\xi}}{\hat{\otimes}}\mathcal{M}_n,{\otimes}_{n\in\bn}U_n )$ is a $\bz_2$-graded von Neumann algebra.\\

In order to arrive at the general form of our product commutation theorem, we start by attacking the case 
of a product of two von Neumann algebras.

\begin{thm}\label{VanDaele}
If $\mathcal{M}\subset\cb(\ch)$ and $\cn\subset\cb(\ck)$ are von Neumann algebras on $\bz_2$-graded Hilbert spaces
$(\ch, U)$ and $(\ck, V)$ such that $U \mathcal{M} U=\mathcal{M}$ and $V \cn V=\cn$, then
$$(\mathcal{M}\,\hat{\otimes} \cn)^\wr=\mathcal{M}^\wr\,\hat{\otimes} \cn^\wr\, .$$
\end{thm}

\begin{proof}
We start with the inclusion $\mathcal{M}^\wr\,\hat{\otimes} \cn^\wr\subset (\mathcal{M}\,\hat{\otimes} \cn)^\wr$, which
can be checked  by direct computation as follows.  Since for any von Neumann algebra $\cl$ by definition one has $\mathcal{\cl}^\wr=\eta(\cl')$,
we need to show that
$$[\eta_U(M')\hat{\otimes}\eta_V(N'),\eta_{U\,{\otimes} V}(M\hat{\otimes} N) ]=0$$
for every homogeneous $ M\in\mathcal{M}, N\in\cn, M'\in\mathcal{M}', N'\in\cn'$.
This requires an easy but tedious inspection of the signs, which we leave out.

The converse implication is obtained as an application of Lemma \ref{Rieffel}
with $\ca=\mathcal{M}\,\hat{\otimes} \cn$ and $\cb=\mathcal{M}^\wr\,\hat{\otimes} \cn^\wr$. First note that without loss of generality we may assume that $\mathcal{M}$ has an even cyclic
vector $\xi_1\in\ch$ and $\cn$ has an even cyclic vector $\xi_2\in\ck$. This can be seen as in
\cite{RvD} and references therein because even normal
states on a graded von Neumann algebra separate its points.
In particular, $\xi:=\xi_1{\otimes}\xi_2$ is an even cyclic vector for $\mathcal{M}\hat{{\otimes}} \cn$.\\
In order to apply Lemma \ref{Rieffel}, we need to make sure that
$\eta(\mathcal{M}\,\hat{\otimes} \cn)_s\xi+ i(\mathcal{M}^\wr\,\hat{\otimes} \cn^\wr)_s\xi$ is dense in $\ch{\otimes}\ck$, where $\eta:=\eta_U\hat{\otimes} \eta_V$.
Now, as is easily checked, $\eta(\mathcal{M}\,\hat{\otimes} \cn)_s\xi \supset \eta_U(\mathcal{M}_s)\xi_1\,{\otimes} \eta_V(\cn_s)\xi_2$ and $(\mathcal{M}^\wr\,\hat{\otimes} \cn^\wr)_s\xi\supset
(\mathcal{M}^\wr)_s\xi_1\,{\otimes} (\cn^\wr)_s\xi_2$, which means
it is enough to verify that
$$\eta_U(\mathcal{M}_s)\xi_1\,{\otimes} \eta_V(\cn_s)\xi_2+i(\mathcal{M}^\wr)_s\xi_1\,{\otimes} (\cn^\wr)_s\xi_2$$
is dense in $\ch\,{\otimes}\ck$. In light
of Remark \ref{realorthogonal}, we are reconducted to
verifying that
$$\eta_U(\mathcal{M}_s)\xi_1\,{\otimes} \eta_V(\cn_s)\xi_2+
i\big([\eta_U(\mathcal{M}_s)\xi_1]^\perp\,{\otimes}[\eta_V(\cn_s)\xi_2]^\perp\big) $$
is dense, which follows from the final lemma in \cite{RvD}.
\end{proof}

We can finally state the general version.

\begin{thm}\label{infinitecomm}
If $\{(\ch_i, U_i):i\in \bn\}$ is a family of $\bz_2$-graded Hilbert spaces, and  $\cn_i\subset\cb(\ch_i)$ are von
Neumann algebras such that $U_i\cn_iU_i=\cn_i$, $\i\in\bn$, then
$$\big(\underset{\boldsymbol{\xi}}{\hat{\otimes}}\cn_i\big)^\wr=\underset{\boldsymbol{\xi}}{\hat{\otimes}}\cn_i^\wr$$
for any sequence $\boldsymbol{\xi}:=\{\xi_i: i\in \bn\}$ of  unit vectors $\xi_i\in\ch_i$ with
$U_i\xi_i=\xi_i$, $i\in \bn$.
\end{thm}

\begin{proof}
First note that a straightforward induction shows that Theorem \ref{VanDaele} holds for any finite Fermi tensor product.
Again, the inclusion $\underset{\boldsymbol{\xi}}{\hat{\otimes}}\cn_i^\wr\subset\big(\underset{\boldsymbol{\xi}}{\hat{\otimes}}\cn_i\big)^\wr$ is trivially
satisfied. 

For the converse inclusion, take $T$ in $\big(\underset{\boldsymbol{\xi}}{\hat{\otimes}}\cn_i\big)^\wr$. We will show that $T$ sits in the weak closure of $\underset{\boldsymbol{\xi}}{\hat{\otimes}}\cn_i^\wr$. Set $\ch:=\underset{\boldsymbol{\xi}}{{\otimes}}\ch_i$.
Now a neighborhood of $T$ for the weak operator topology is of the form
 $$\mathcal{G}=\{S\in\cb(\ch): |\langle (T-S)x_i, y_i \rangle|<\eps, i=1, 2, \ldots, n\},$$ 
for
some $x_i, y_i\in\ch$, $i=1, 2, \ldots, n$, and $\eps>0$.
By definition of $\ch$, there exists $N\in\bn$ such that
$$\|Px_i-x_i\|\leq \varepsilon\,\,{\rm and}\,\, \|Py_i-y_i\|\leq\eps,\, \textrm{for every}\,\,i=1, 2, \ldots, n,$$
where
$P$ is the projection uniquely determined on simple tensors $\otimes_{i\in\bn} u_i$ in $\ch$ by
$$P({\otimes} u_i)={\otimes}_{i=1}^N u_i{\otimes}
({\otimes}_{i\geq N+1}\langle u_i, \xi_i\rangle\xi_i)\, .$$
The same calculations as in the proof of Proposition 9 on page 34 of \cite{Gui2}
show that $PT$ lies in $(\hat{\otimes}_{i=1}^N\cn_i)^\wr\,\hat{\otimes}\bc\,\hat{\otimes}\bc\cdots$. Since we have
$$(\hat{\otimes}_{i=1}^N\cn_i)^\wr\,\hat{\otimes}\bc\,\hat{\otimes}\bc\cdots=
\hat{\otimes}_{i=1}^N\cn_i^\wr\,\hat{\otimes}\bc\,\hat{\otimes}\bc\cdots\subset
\underset{\boldsymbol{\xi}}{\hat{\otimes}}\cn_i^\wr$$
the thesis will be arrived at as long as we make sure that $PT\in\mathcal{G}$.
This follows exactly as in the above reference.
\end{proof}

As an easy application of the theorem above, we provide the following result, where pureness of product states is
addressed.

\begin{prop}\label{pure}
Let $(\ga_i, \theta_i)$ be $\bz_2$-graded $C^*$-algebras, and let $\om_i\in\cs(\ga_i)$ be pure states, $i\in\bn$.
Suppose
all of these states are even but one, say $\om_1$. If $\pi_{\om_1}$ and $\pi_{\om_1\circ\theta}$
are unitarily equivalent, then the product state $\times_i \om_i$ is pure as well.
\end{prop}

\begin{proof}
It suffices to note that under the above hypotheses $\pi_{\times_i \om_i}$
is still (unitarily equivalent to) $\underset{\boldsymbol{\xi}}{\hat{\otimes}}\pi_{\om_i}$
with $\boldsymbol{\xi}:=\{\xi_{\om_i}\}_{i\in\bn}$, see Proposition \ref{ArMor}, which means
Theorem \ref{infinitecomm} applies.
\end{proof}

We are going to further apply Theorem \ref{infinitecomm} to infer factoriality
of an infinite product of even factorial  states. 
To do so, we first establish a couple of related results.

\begin{lem}\label{union}
Let $(\ch, U)$ be a $\bz_2$-graded Hilbert space and let $\cn_1, \cn_2\subset \cb(\ch)$ be
von Neumann algebras with $U\cn_iU=\cn_i$, $i=1, 2$, then
$$(\cn_1\cap\cn_2)^\wr=\cn_1^\wr\vee\cn_2^\wr\quad {\rm and}\quad (\cn_1\vee\cn_2)^\wr=\cn_1^\wr\cap\cn_2^\wr\, .$$
\end{lem}

\begin{proof}
The first equality is arrived at through the chain of equalities below
\begin{align*}
(\cn_1\cap\cn_2)^\wr&=\eta_U((\cn_1\cap\cn_2)')=\eta_U(\cn_1'\vee\cn_2')\\
&=\eta_U(\cn_1')\vee\eta_U(\cn_2')=\cn_1^\wr\vee\cn_2^\wr\, .
\end{align*}
The second follows analogously.
\end{proof}

\begin{prop}
Let $(\ch, U)$ and $(\ck, V)$ be $\bz_2$-graded Hilbert spaces. If
$\mathcal{M}_i\subset\cb(\ch)$ and $\cn_i\subset\cb(\ck)$, $i=1,2$, are von Neumann algebras
such that $U\mathcal{M}_iU=\mathcal{M}_i$ and $V\cn_i V=\cn_i$, $i=1, 2$, then
$$(\mathcal{M}_1\,\hat{\otimes}\cn_1)\cap (\mathcal{M}_2\,\hat{\otimes}\cn_2) =
(\mathcal{M}_1\cap\mathcal{M}_2)\hat{\otimes}(\cn_1\cap\cn_2)
$$ and
$$(\mathcal{M}_1\,\hat{\otimes}\cn_1)\vee (\mathcal{M}_2\,\hat{\otimes}\cn_2) =
(\mathcal{M}_1\vee\mathcal{M}_2)\hat{\otimes}(\cn_1\vee\cn_2)\, .
$$
\end{prop}

\begin{proof}
As for the first equality, only the inclusion
$$(\mathcal{M}_1\,\hat{\otimes}\cn_1)\cap (\mathcal{M}_2\,\hat{\otimes}\cn_2)\subset
(\mathcal{M}_1\cap\mathcal{M}_2)\hat{\otimes}(\cn_1\cap\cn_2)\
$$
needs to be dealt with, for the converse inclusion is trivially verified.\\
To this aim, we show that
$$\big((\mathcal{M}_1\cap\mathcal{M}_2)\hat{\otimes}(\cn_1\cap\cn_2)\big)^\wr\subset\big((\mathcal{M}_1\,\hat{\otimes}\cn_1)\cap (\mathcal{M}_2\,\hat{\otimes}\cn_2)\big)^\wr\, .
$$
Now by Theorem \ref{VanDaele} and Lemma \ref{union} we have
\begin{align*}
\big((\mathcal{M}_1\cap\mathcal{M}_2)\hat{\otimes}(\cn_1\cap\cn_2)\big)^\wr&=
(\mathcal{M}_1\cap\mathcal{M}_2)^\wr\hat{\otimes}(\cn_1\cap\cn_2)^\wr\\
&=(\mathcal{M}_1^\wr\vee\mathcal{M}_2^\wr)\hat{\otimes}(\cn_1^\wr\vee\cn_2^\wr)\\
&\subset (\mathcal{M}_1^\wr\,\hat{\otimes}\cn_1^\wr)\vee (\mathcal{M}_2^\wr\,\hat{\otimes}\cn_2^\wr)\\
&=(\mathcal{M}_1\hat{\otimes} \cn_1)^\wr\vee(\mathcal{M}_2\hat{\otimes} \cn_2)^\wr\\
&=\big((\mathcal{M}_1\hat{\otimes} \cn_1)\cap (\mathcal{M}_2\hat{\otimes} \cn_2)\big)^\wr\, .
\end{align*}
In the second equality both inclusions can be verified directly.
\end{proof}
\begin{rem}\label{rem4.9}
By using Theorem \ref{infinitecomm} one sees that the first equality of the above result
holds with infinite graded tensor products as well.
\end{rem}

Before stating our next result, we recall that a factor is a von Neumann algebra with trivial center.

\begin{prop}\label{factor}
Under the same hypotheses as in Theorem \ref{infinitecomm}, an infinite Fermi tensor product is
 a factor if and only if each component is a factor.
\end{prop}

\begin{proof}
Set $\mathcal{R}:=\underset{\boldsymbol{\xi}}{\hat{\otimes}}\,\mathcal{R}_i$, where
the $\mathcal{R}_i$'s are all factors.  With $U=\otimes_{i\in\bn} U_i$, 
thanks to Theorem \ref{infinitecomm} and Remark \ref{rem4.9} we have
\begin{align*}
\eta_U(\mathcal{R})\cap \mathcal{R}^\wr=\underset{\boldsymbol{\xi}}{\hat{\otimes}}\, \eta_{U_i}(\mathcal{R}_i)\cap \underset{\boldsymbol{\xi}}{\hat{\otimes}}\, \mathcal{R}_i^\wr=
\underset{\boldsymbol{\xi}}{\hat{\otimes}}\,
 \big(\eta_{U_i}(\mathcal{R}_i)\cap \mathcal{R}_i^\wr\big)=\bc\,,
\end{align*}
which shows that $\mathcal{R}$ is still a factor. The converse implication is obvious.
\end{proof}

A representation $\pi: \ga\rightarrow\cb(\ch)$ of a given $C^*$-algebra
is said to be factorial if $\pi(\ga)''$ is a factor, {\it i.e.}
$\pi(\ga)''\cap\pi(\ga)'=\bc 1$. A state $\varphi$ of a $C^*$-algebra is factorial if its
GNS representation is. Moreover, the type of a factorial state is by definition the same as the type of the factor
generated by its GNS representation.

\begin{prop}
Let $(\ga_i, \theta_i)$ be $\bz_2$-graded $C^*$-algebras, and let $\om_i$ be  in $\cs_+(\ga_i)$,  $i\in\bn$.
The product state $\om=\times_i \om_i$ is factorial if and only if each $\om_i$ is.
\end{prop}

\begin{proof}
A straightforward application of Remark \ref{prodrep} and Proposition \ref{factor}.
\end{proof}

Actually, far more can be said about the type of factor one can obtain from a GNS representation as above.
In fact, the analysis  conducted in \cite{St} for tensor products carries over  almost
\emph{verbatim} to the graded case. More precisely, we can provide a graded version of Theorem 2.2 in \cite{St}.
We limit ourselves to stating the result  since the proof is exactly the same as the original by St\o rmer.

\begin{prop}\label{typefactor}
If $\om$ is an even factorial state on a $\bz$-graded $C^*$-algebra $(\ga, \theta)$, then
\begin{itemize}
\item [(i)] $\times^\bn \om$ is of type $I_1$ if and only if  $\om$ is mutiplicative;
\item [(ii)] $\times^\bn \om$ is of type $I_\infty$ if and only if $\om$ is pure but is not multiplicative;
\item [(iii)] $\times^\bn \om$ is of type $II_1$ if and only if $\om$ is a trace but is not multiplicative.
\item [(iv)] $\times^\bn \om$ is of type $II_\infty$ if and only if the restriction of the vector state $\varphi_{\xi_\om}$ to
$\pi_\om(\ga)'$ is a trace, and $\om$ is neither pure nor a trace.
\item [(v)] $\times^\bn \om$ is of type $III$ if and only if the restriction of the vector state $\varphi_{\xi_\om}$ to
$\pi_\om(\ga)'$ is not a trace.
\end{itemize}

\end{prop}


\section*{Acknowledgments}
The authors acknowledge Italian INDAM-GNAMPA.


\begin{thebibliography}{99}

\bibitem{ABCL} Accardi L., Ben Ghorbal A., Crismale V., Lu Y.G. {\it Singleton conditions and quantum de Finetti's theorems}, Infin. Dimens. Anal. Quantum Probab. Relat. Top.
\textbf{11} (2008), 639--660.

\bibitem{AL} Accardi L., Lu Y. G. {\it A continuous version of de Finetti's theorem}, Ann. Probab. \textbf{21} (1993), 1478--1493.







\bibitem{A} Araki H. \emph{Mathematical Theory of Quantum Fields}, International Series of Monographs on Physics, 101, Oxford University Press, 2009.
\bibitem {AM1} Araki H., Moriya H. \emph{Joint extension of states of subsystems for a CAR system},
Commun. Math. Phys. \textbf{237} (2003), 105--122.
\bibitem{Av} Avitzour D. {\it Free products of $C^*$-algebras},  Trans. Am. Math. Soc.  {\bf 271} (1982), 423--435.





\bibitem{BR1} Bratteli O.,  Robinson D.W. \emph{Operator algebras and
quantum statistical mechanics 1}, 2nd edition, Springer-Verlag, New York, 1997.


\bibitem{CDF} Crismale V., Duvenhage R., Fidaleo F.
{\it $C^*$-fermi systems and detailed balance}, Anal. Math. Phys.
\textbf{11} (2021), Paper No. 11.

\bibitem{CFCMP} Crismale V., Fidaleo F.  {\it de Finetti theorem on the CAR algebra},
Commun. Math. Phys. \textbf{315} (2012), 135--152.

\bibitem{CrFid} Crismale V., Fidaleo F.
{\it Exchangeable stochastic processes and symmetric states in quantum probability},
Ann. Mat. Pura Appl., {\bf 194} (2015), 969--993.

\bibitem{CrFid2} Crismale V., Fidaleo F. {\it Symmetries and ergodic
properties in quantum probability}, Colloq. Math. \textbf{149} (2017), 1--20.

\bibitem{CRZ} Crismale V., Rossi S., Zurlo P. {\it On $C^*$-norms on Fermi tensor products}, to appear in 
Banach J. Math. Anal., arXiv preprint 2112.03988.


\bibitem{Dix} Dixmier J. \emph{$C^*$-algebras}, North Holland, Amsterdam 1977.

\bibitem{Gui2} Guichardet A., \emph{Tensor products of $C^*$-algebras Part II. Infinite tensor products}, Lecture
Notes Series No. 13, Aahrus Universitet.
\bibitem{HS} Hewitt E., Savage L.F. \emph {Symmetric measures on Cartesian products} Trans. Am. Math. Soc. \textbf{80}, 470--501.
(1955)















\bibitem{Ka} Kallenberg O. {\it Probabilistic Symmetries and Invariance principles}, Springer, Berlin 2005.

\bibitem{K} K\"{o}stler C. {\it A noncommutative extended De Finetti theorem} J. Funct. Anal. {\bf 258} (2010), 1073--1120.





\bibitem{RvD} Rieffel M., Van Daele A., \emph{The commutation theorem
for tensor products of von Neumann algebras}, Bull. London. Math. Soc. \textbf{7} (1975), 257--260.

\bibitem{Sak}  Sakai S. \emph{$C^*$-Algebras and $W^*$-Algebras}, Springer, Berlin 1971.

\bibitem{St67} St\o rmer E. \emph{Large Groups of Automorphims of $C^*$-Algebras}, Commun. Math. Phys.
\textbf{5} (1967), 1--22.
\bibitem{St} St\o rmer E. \emph{Symmetric States of infinite Tensor Products of C*-algebras.}
J. Funct. Anal. \textbf{3} (1969), 48--68.

\bibitem{T1} Takesaki M.
{\it Theory of operator algebras I}, Springer, Berlin--Heidelberg--New
York 1979.

\bibitem{T} Tomita M. {\it Standard forms of von Neumann algebras}, Vth functional analysis simposium of the mathematical society of Japan, Sendai 1967.



\end{thebibliography}
\end{document}